\documentclass[11pt,reqno]{amsart}
\usepackage[utf8]{inputenc}
\usepackage{amsmath,amsthm,amssymb,physics}
\usepackage[vcentermath,enableskew]{youngtab}
\usepackage{tikz}
\usepackage{geometry}
\usepackage[colorlinks,breaklinks]{hyperref}
\usepackage{tikz-cd}
\usepackage{mathrsfs}

\title{On the action of the Weyl group on canonical bases}

\author{Fern Gossow}
\address{F.~Gossow: School of Mathematics and Statistics, University of Sydney, Australia}
\email{F.Gossow@maths.usyd.edu.au}

\author{Oded Yacobi}
\address{O.~Yacobi: School of Mathematics and Statistics, University of Sydney, Australia}
\email{oded.yacobi@sydney.edu.au}

\date{}

\def\SYT{\mathrm{SYT}}
\def\Res{\mathsf{Res}}

\DeclareMathOperator{\End}{End}
\DeclareMathOperator\sh{sh}

\DeclareMathOperator{\Span}{span}
\def\pr{\mathsf{pr}}
\def\ev{\mathsf{ev}}
\def\nev{\mathsf{nev}}

\def\Z{\mathbb Z}

\def\KL{\mathbb{K}\mathbb{L}}

\def\H{\mathcal{H}}
\def\GL{\mathrm{GL}}

\def\wt{\mathsf{wt}}
\def\SSYT{\mathrm{SSYT}}
\def\SYT{\mathrm{SYT}}
\def\rect{\mathsf{rect}}

\newcommand{\cA}{\mathcal{A}}

\newcommand{\bB}{\mathbb{B}}
\newcommand{\cB}{\mathcal{B}}

\newcommand{\cC}{\mathcal{C}}

\newcommand{\sC}{\mathscr{C}}

\newcommand{\cD}{\mathcal{D}}

\newcommand{\sE}{\mathsf{E}}

\newcommand{\sF}{\mathsf{F}}

\newcommand{\cH}{\mathcal{H}}

\newcommand{\calZ}{\mathcal{Z}}

\newcommand{\cL}{\mathcal{L}}

\newcommand{\rL}{\mathrm{L}}

\newcommand{\cO}{\mathcal{O}}

\newcommand{\cR}{\mathcal{R}}

\newcommand{\cS}{\mathcal{S}}

\newcommand{\cT}{\mathcal{T}}

\newcommand{\cX}{\mathcal{X}}

\newcommand{\fg}{\mathfrak{g}}

\newcommand{\fh}{\mathfrak{h}}

\newcommand{\bk}{\mathbf{k}}

\newcommand{\br}{\mathbf{r}}

\newcommand{\bs}{\mathbf{s}}

\newcommand{\bt}{\mathbf{t}}

\newcommand{\BB}{\mathbb{B}}
\newcommand{\CC}{\mathbb{C}}

\newcommand{\LL}{\mathbb{L}}

\renewcommand{\SS}{\mathbb{S}}

\newcommand{\ZZ}{\mathbb{Z}}

\newcommand{\bla}{\boldsymbol{\lambda}}

\renewcommand{\phi}{\varphi}
\renewcommand{\emptyset}{\varnothing}

\renewcommand{\tilde}[1]{\widetilde{#1}}
\newcommand{\ol}[1]{\overline{#1}}
\newcommand{\ul}[1]{\underline{#1}}

\DeclareMathOperator{\res}{res}

\newcommand{\fsl}{\mathfrak{sl}}

\newcommand{\sone}{\mathsf{1}}
\newcommand{\mmod}{\mathrm{-mod}}
\newcommand{\Pmod}{\mathrm{-pmod}}

\newcommand\into\hookrightarrow
\newcommand\onto\twoheadrightarrow
\newcommand{\Hom}{\operatorname{Hom}}
\newcommand{\nc}{\newcommand}
\nc{\la}{\lambda}
\nc{\Iso}{\mathsf{Iso}}
\nc{\Irr}{\mathsf{Irr}}
\nc{\Id}{\mathrm{Id}}

\DeclareMathOperator\Des{Des}

\numberwithin{equation}{section}
\newtheorem{Theorem}[equation]{Theorem}
\newtheorem{Proposition}[equation]{Proposition}
\newtheorem{Lemma}[equation]{Lemma}
\newtheorem{Corollary}[equation]{Corollary}
\newtheorem{Conjecture}[equation]{Conjecture}
\newtheorem{Definition}[equation]{Definition}
\newtheorem{Question}[equation]{Question}
\newtheorem{Remark}[equation]{Remark}
\theoremstyle{definition}
\newtheorem{Example}[equation]{Example}

\begin{document}

\begin{abstract}
We study representations of simply-laced Weyl groups which are equipped with canonical bases.  Our main result is that for a large class of representations, the separable elements of the Weyl group $W$ act on these canonical bases by bijections up to lower-order terms. Examples of this phenomenon include the action of separable permutations on the Kazhdan--Lusztig basis of irreducible representations for the symmetric group, and the action of separable elements of $W$ on dual canonical bases of weight zero in tensor product representations of a Lie algebra.  Our methods arise from categorical representation theory, and in particular the study of the perversity of Rickard complexes acting on triangulated categories.\\
\smallskip\noindent
{\bf Keywords.} Weyl groups, representation theory, canonical bases, categorification
\end{abstract}

\maketitle

\section{Introduction}

\subsection{Set-up}
Let $I$ be a simply-laced Dynkin diagram of finite type, and $W=W_I$ its associated Weyl group.  In this paper we'll be interested in the following situation: we have a (complex) representation $\pi:W \to \GL(U)$ equipped with a canonical basis $\bB \subset U$,  usually of geometric or categorical origin.

We ask  the following:
\begin{Question}\label{quest:1}
Can we extract interesting symmetries of $\bB$ from the action of $w\in W$?
\end{Question} 
The origins of this question go back to the 1990s, when the longest element $w_0\in W$ was studied in this context: Lusztig investigated its action on the canonical basis elements of weight zero in irreducible representations of the Lie algebra \cite{LusCanBasesArising}, Mathas studied its action on the Kazhdan--Lusztig basis of the Hecke algebra \cite{M}, and Berenstein--Zelevinsky and Stembridge studied the case of the longest permutation acting on the Kazhdan--Lusztig basis of irreducible Specht modules \cite{BZ,S}.  More recently, Rhoades studied the case of a particular Coxeter element in the symmetric group (the ``long cycle'') acting on Specht modules corresponding to rectangular partitions \cite{R}.

The main aims of our work are two-fold.  Firstly, we show that these results are very special cases of a much wider phenomenon: for a large class of elements (the ``separable'' elements of $W$) and for a large class of representations, the action of $w$ encodes interesting symmetries of $\bB$.  These symmetries arise from an underlying crystal structure in the sense of Kashiwara \cite{Kash91}.  Secondly, our methods are rooted in categorical representation theory, and in particular we show how perverse autoequivalences of triangulated categories can be well-suited to study and explain interesting combinatorial phenomena arising in Lie theory.   

\subsection{The symmetric group}\label{sec:symmetric-group}
We consider first the prototypical examples of the kinds of results we have in mind.  Set $I=\mathrm{A}_{n-1}$, i.e. $W=S_n$ is the symmetric group, and for a partition $\lambda \vdash n$ let $S^\lambda$ be the corresponding Specht module. Recall that as we range over all partitions, these give a complete list of the irreducible representations of $S_n$.  

The Kazhdan--Lusztig (KL) basis of the Hecke algebra (cf.\ Appendix \ref{app:KL}) descends to a canonical basis of the representations $S^\lambda$ (cf.\ Section \ref{sect:KL}).  We denote this basis by $\KL_\lambda=\{ C_T \;|\; T \in \SYT(\lambda) \}$, where $\SYT(\lambda)$ is the set of standard Young tableaux of shape $\lambda$.  Then, by Berenstein--Zelevinsky and Stembridge \cite{BZ,S}, for any $\lambda \vdash n$ and any $T \in \SYT(\lambda)$, we have $
w_0\cdot C_T =\pm C_{\ev(T)}$, where $\ev: \SYT(\lambda) \to \SYT(\lambda)$ is the evacuation operator (cf.\ Example \ref{ex:evac}).  Thus  Question \ref{quest:1} has an affirmative answer in the case when $w$ is the longest element acting on the KL basis of a Specht module.

Rhoades  proved a similar result for the ``long cycle'' $c_n=(1,2,\hdots,n) \in S_n$:  in the case when $\lambda$ is rectangular, $c_n$ acts on the KL basis by the promotion operator $\pr$, another well-studied symmetry of  $\SYT(\lambda)$ \cite{R}. Using this, Rhoades established a cyclic sieving phenomenon for the action of promotion on rectangular tableaux.

It's  easy to see  that Rhoades' Theorem is false for non-rectangular partitions.  Nevertheless, we were recently able to show that the theorem holds for all partitions if one allows for a more general framework \cite{GY}.  Indeed, let $\lambda \vdash n$ be any partition, and define a preorder on $\SYT(\la)$ by $R \leq T$ if the box of $R$ containing $n$ is in a higher row than the box of $T$ containing $n$.  We proved that $c_n \cdot C_T=\pm C_{\pr(T)}$ \textit{up to lower-order terms}.  

Upon reviewing these results, it is natural to try to refine Question \ref{quest:1}. For this we define:

\begin{Definition}\label{def:lot}
    Let $U$ be a vector space with basis $\BB$, $w:U \to U$ an operator, $\xi:\BB \to \BB$ a bijection, and $\leq$ a preorder on $\BB$.  Then $w$ acts on $(U,\BB,\leq)$ by $\xi$ \emph{up to lower-order terms (l.o.t.)} if for every $x \in \BB$ there are $a_y \in \ZZ$ such that:
    $$w(x)=\pm \xi(x) +\sum_{y< x} a_y\xi(y),$$
    and the sign of $\xi(x)$ only depends on the equivalence class of $x$ under $\leq$. 
\end{Definition}
In some cases, we will leave the bijection and preorder unspecified and simply say that $w$ acts on $(U,\BB)$ by some bijection up to lower-order terms. This is reasonable since the bijection is uniquely specified by the action of $w$ on $\BB$ (cf.\ Lemma \ref{lem:lot-unique}). There is also a compact description of the above definition in terms of QR decompositions (cf.\ Lemma \ref{lem:QR}).

In light of this definition, we now ask:
\begin{Question}\label{quest:2}
Which permutations $w \in S_n$ act on $(S^\lambda,\KL_\la)$ by some bijection up to l.o.t.? Under the association of $\KL_\la$ with $\SYT(\la)$, how can we describe the bijections which appear in terms of tableaux?
\end{Question}

Our first main result is that a large and well-studied class of permutations, namely the separable permutations, act on the KL basis by bijections up to l.o.t.  

Separable permutations are constructed as follows.
Given  $x \in S_m$, $y\in S_n$ with permutation matrices $[x],[y]$, define $x \oplus y,x \ominus y \in S_{m+n}$ by:
\begin{align*}
[x \oplus y] = \begin{pmatrix}
[x] & 0 \\
0 & [y] 
\end{pmatrix} ,\quad\quad [x \ominus y] = \begin{pmatrix}
0 & [x]  \\
[y] & 0 
\end{pmatrix}.
\end{align*} 
The set $S_n^{\mathrm{sep}}$ of separable permutations is given by those permutations obtainable from $(1) \in S_1$ by repeated applications of $\oplus$ and $\ominus$ \cite{Kitaev}.
In particular $S_n^\text{sep}$ contains both $w_0$ and $c_n$.  Note that $S_n^{\text{sep}}$ can be alternatively defined as the set of permutations avoiding the patterns $2413$ and $3142$, and they are counted by the large Schr\"oder numbers\footnote{Large Schr\"oder numbers appear as sequence A006318 in the OEIS: $1, 2, 6, 22, 90, 394, 1806, 8558, \hdots$}.  

We prove:
\begin{Theorem}\label{thm:1}
Fix $w \in S_n^{\mathrm{sep}}$ and $\la \vdash n$ arbitrary.  Then $w$ acts on $(S^\lambda,\KL_\lambda)$ by a bijection up to lower-order terms. 
\end{Theorem}

We discuss this theorem in detail in Section \ref{sect:KL}, where in particular we describe the bijection and preorder associated to $w$ through the combinatorics of tableaux. For an example of the new type of results that arise from this, in Section \ref{sec:KLexamples}  we specify an element $g_n \in S_n$ which is the product of long elements for the tower of symmetric groups $S_n \supset S_{n-1} \supset \cdots\supset S_1$. We show that $g_n$ acts on $(S^\lambda,\KL_\lambda)$ up to l.o.t. by an operation we term ``nested evacuation''.   

We  note that non-separable permutations don't appear to encode any permutations of standard Young tableaux (cf. Example \ref{ex:nonseparable}).  We conjecture that in a precise sense Theorem \ref{thm:1} is the best possible solution to Question \ref{quest:2} (cf. Conjecture \ref{conj:nonseparable}).  If true, this gives a new and very different type of characterisation of separable permutations. 

\subsection{Other Weyl groups}\label{sec:other-weyl-groups}
When $W$ is a simply-laced Weyl group, we use categorical methods to study a much broader class of $W$-modules arising from representations of Lie algebras.  To describe this, let $I$ be the Dynkin diagram of $W$, and for a subdiagram $J \subseteq I$, let $W_J \subseteq W$ denote the parabolic subgroup.  Let $w_J \in W_J$ denote the longest element.  We define the set of separable elements  of $W$ by
\[W^\mathrm{sep}:=\{w_{I_1}\cdots w_{I_r}\in W\mid I \supseteq I_1 \supseteq \cdots \supseteq I_r\}.\]
\begin{Remark}
By splitting each $I_j$ into connected components, commuting disjoint factors and cancelling appearances of $w_J^2$, we are able to rewrite separable elements as a product over $w_J$ for $J\in\mathcal{N}$, where $\mathcal{N}$ is a nested set of $I$ as in \emph{\cite[Theorem 7]{GG2}}. Hence, our characterisation of separable elements matches the definition given by Gaetz and Gao \emph{\cite{GG1}}. In particular, when $W=S_n$, then $W^\mathrm{sep}$ is the set of separable permutations.
\end{Remark}

We'll now exhibit the above phenomenon in $W$-modules arising from representations of Lie algebras, such as tensor products of irreducible highest weight modules. 
We prove that separable elements of $W$ act by bijection up to lower-order terms on the canonical bases of weight zero.  Moreover, the canonical bases of such representations have a natural crystal structure, and the bijection and preorder have an explicit crystal-theoretic description.

Let $\fg$ be the complex semisimple Lie algebra with diagram $I$, and let $X^+$ be the set of dominant integral weights.  For $\lambda \in X^+$ we let $L(\lambda)$ be the irreducible representation of $\fg$ with highest weight $\lambda$.  Given a list of dominant weights $\ul{\la}=(\la_1,\hdots,\la_n)$, let $L(\ul{\la})=L(\lambda_1)\otimes \cdots \otimes L(\lambda_n)$ be the tensor product representation.

Tensor product representations $L(\ul{\la})$ carry various canonical bases, constructed originally by Lusztig \cite{LusTensProd}, which possess amazing positivity properties.  In this introduction we'll focus on the dual canonical basis, but similar results hold also for the canonical basis (cf. Section \ref{sec:dualbasis}).

Recall that there is a natural action of $W$ on the zero weight space $L(\ul{\la})_0$.  Let $\BB(\ul{\la}) \subset L(\ul{\la})_0$ denote the set of dual canonical basis elements of weight zero.  

\begin{Theorem}\label{thm:2}
Let $w \in W^{\mathrm{sep}}$ and let $\ul{\la}$ be a sequence of dominant weights. Then $w$ acts on $(L(\ul\la)_0,\BB(\ul{\la}))$ by a bijection up to lower-order terms.  
\end{Theorem}

We discuss this in detail in Section \ref{sect:tensor-product}.  The bijection and preorder are defined using the  crystal structure on the set of dual canonical basis elements.  We emphasise that as representations of $W$, the spaces $L(\ul{\la})_0$ are very mysterious in general, even in type A  and for $n=1$.

\subsection{Methodology} 

Our methods come from categorical representation theory.  
The starting point is work of Khovanov--Lauda and Rouquier \cite{KL,R} who developed the foundations of this theory for $U(\fg)$ and the quantum group $U_q(\fg)$.  

The definition of a categorical representation is recalled in Section \ref{sec:catrepth}.  For the present discussion, the salient points are as follows.  We will consider representations of $U(\fg)$ on an abelian category $\cC$.  By definition, the category is a direct sum of full subcategories  $\cC_\mu$ indexed by weights, just as a finite dimensional representation is a direct sum of weight spaces.  From the categorical structure, the complexified Grothendieck group $V=[\cC]_\CC$ is  endowed with an action of $U(\fg)$, and the classes of simple objects in $\cC$ descend to a distinguished basis of $V$.  In the cases of interest to us, this is the dual canonical basis.

Chuang and Rouquier constructed the Rickard complex, which is a functor of bounded derived categories: $\Theta_i:D^b(\cC_\mu) \to D^b(\cC_{s_i(\mu)})$.  Here  $s_i \in W_I$ is the simple reflection associated to a node $i \in I$.  They proved that $\Theta_i$ is an equivalence of triangulated categories, and used this to prove Brou\'e's conjecture for symmetric groups \cite{CR}. Many important equivalences in algebraic geometry and related areas turn out to be Rickard complexes, for example the spherical twists that appear in Mirror Symmetry.

The autoequivalences $\Theta_i$ on $D^b(\cC_\mu)$ satisfy the braid relations (up to isomorphism) \cite{CK}.  Hence, given $w \in W_I$ we have an autoequivalence $\Theta_w:D^b(\cC_\mu) \to D^b(\cC_{w(\mu)})$, which is defined up to isomorphism.

Focusing now on the case $\mu=0$, the connection to our work arises since the autoequivalence $\Theta_w$ categorifies the operator $w:V_0 \to V_0$.  Therefore, to study the action of $w$ on canonical bases of $V_0$, we'd like to understand the image of simple objects in $\cC_0$ under $\Theta_w$.  

It turns out that the theory of perverse equivalences, also due to Chuang and Rouquier \cite{CR2}, provides the correct setting to extract this information.  Indeed, a perverse equivalence is essentially a categorification of Definition \ref{def:lot} (cf.\ Lemma \ref{lem:lot-general}). 
More precisely, an equivalence $\Theta:D^b(\cC) \to D^b(\cC')$ is perverse if it preserves some filtrations on the domain and codomain categories such that on each subquotient the induced equivalence is t-exact up to shift.  Crucially, from a perverse equivalence one can extract a bijection of the simple objects by keeping track of the t-exact equivalences on each subquotient.  The details of this theory are explained in Section \ref{sec:pervdefs}.  

The main result that underlies our work is the following.

\begin{Theorem}\label{thm:3}
Let $\cC$ be a categorical representation of $U(\fg)$ and let $w \in W^{\mathrm{sep}}$ be a separable element.  Then $\Theta_w$ is a perverse equivalence.
\end{Theorem}  

The case $w=w_I$ was studied in \cite{HLLY}.  There is a subtlety in extending this case to all separable elements, since compositions of perverse equivalences are not in general perverse.  In Section \ref{sec:pervysepy}, we develop the necessary techniques to prove Theorem \ref{thm:3}, and in doing so we precisely describe the necessary filtrations and resulting bijections.  By investigating the consequences of this theorem on the level of Grothendieck groups, we obtain (generalisations of) Theorems \ref{thm:1} and \ref{thm:2}.

\subsection{Organisation of paper}
We give a brief overview of the paper:

In Section \ref{sec:2} we recall some of the basic representation-theoretic objects which play a key role throughout the paper.  Particularly important is the notion of an ordered $I$-filtration (Definition \ref{def:I-filt}), which is a categorical analogue of an isotypic filtration.  This will be crucially used in the construction of perverse equivalences that follow.

Section \ref{sec:3} is devoted to the theory of perverse equivalences.  The main result is Theorem \ref{thm:mainthm}, which is the precise formulation of Theorem \ref{thm:3} above.  

In Section \ref{sect:crystal} we explain how to deduce combinatorial consequences from perversity.  Theorem \ref{thm:gencomb} states a general result: in any $W$-module which arises from categorification (cf. Equation \eqref{eq:Wmods}), separable elements act by bijections up to l.o.t.  We also explain how to deduce analogous results about dual bases (cf. Corollary \ref{cor:gencomb}).

In Section \ref{sect:KL}, we prove  Theorem \ref{thm:1} as a special case of Theorem \ref{thm:2}. We then develop the combinatorics of type A crystals in order to find explicit descriptions for $\xi$ and $\leq$ on $\KL_\lambda$ in terms of tableaux. We  consider a number of special cases, including the generalisation of Rhoades' Theorem concerning the action of the long cycle (cf. Proposition \ref{prop:Rhoades}).

In Section \ref{sect:tensor-product}, we consider   canonical bases of tensor product representations. We give several examples of Theorem \ref{thm:2}, including the action of the long element of $S_n$ on the dual canonical basis in the $d$-fold tensor power of the standard representation, and the actions of separable permutations $(1,3),(1,2,3) \in S_3$ on the canonical and dual canonical basis of the tensor square of the adjoint representation of $\fsl_3$.

Appendix \ref{app:QR} gives a number of general results about operators which act by bijections up to l.o.t., including the connection to QR decompositions. In Appendix \ref{app:KL}, we generalise the definition of separable elements to arbitrary Coxeter groups, and extend known results about the Kazhdan--Lusztig cell modules to prove an analogue of Theorem \ref{thm:1} for the left regular representation of a Coxeter group (cf. Theorem \ref{thm:lot-KL}).  This argument does not require categorical methods.

\subsection*{Acknowledgements} We thank Ben Elias, Christian Gaetz, Anthony Licata, Daniel Tubbenhauer and Noah White for helpful conversations which improved this work. We are grateful to an anonymous referee for their feedback and suggestions. O.Y. is supported by the Australian Research Council under DP230100654.

\section{Representations, crystals and categorifications}\label{sec:2}

In this section we recall the necessary background about representation theory (classical and categorical), and associated combinatorial objects which we will need in the rest of the paper.  This section also serves to fix notation.

\subsection{Basic objects} \label{sec:basics}

Let $I$ be a simply-laced Dynkin diagram of finite type. For nodes $i,j \in I$ we write $i \sim j$ if they are connected by an edge, and $i \not\sim j$ otherwise. 

Let $\fg=\fg_I$ be the complex semisimple Lie algebra associated to $I$, with Chevalley generators $e_i,f_i,h_i$ for $i\in I$, and Cartan subalgebra $\fh_I=\Span_\CC(\{h_i\mid i \in I\})$.  We denote by $X=X_I$ its integral  weight lattice, and let $\langle h,\la \rangle \in \ZZ$ denote the natural pairing between elements of the integral Cartan subalgebra (i.e. integral linear combinations of the $h_i$) and integral weights.

Let $X^+=X^+_I$ be the cone of integral dominant weights, and $R=R_I \subset X$ the root lattice.  The latter is a free $\ZZ$-module generated by simple roots $\alpha_i$ for $i \in I$.  For $\alpha=\sum_i n_i\alpha_i \in R$ we let $\mathrm{ht}(\alpha):=\sum_i n_i$.  The simple roots define a partial order $\preceq$ (denoted $\preceq_{I}$) on $X$ given by $\mu \preceq_{I} \mu' \Longleftrightarrow \mu' -\mu \in \sum_{i\in I}\mathbb{Z}_{\geq0}\alpha_i$.

We will be working exclusively with complex finite-dimensional representations of $\fg$.  For such a representation $V$ and for $\mu \in X$, we let $V_\mu$ denote its $\mu$ weight space, so that $V=\bigoplus_{\mu \in X}V_\mu$ is the weight decomposition of $V$.
For $\la \in X^+$ we let $L(\la)$ be the irreducible representation with highest weight $\la$.  For any representation $V$, there is a canonical isomorphism $$\bigoplus_{\la \in X^+} \big(L(\la) \otimes \Hom_\fg(L(\la),V)\big) \longrightarrow V,$$ and we let $\Iso_\la(V)$ denote the image of the $\la$-summand under this map.  This is the $\la$-isotypic component of $V$, and we say $V$ is an isotypic representation of highest weight $\la$ if $V=\Iso_\la(V)$.

Given a subdiagram $J \subseteq I$, we consider the Lie subalgebra $\fg_J \subseteq \fg$.  For a representation $V$ of $\fg$, we let $\Res_J(V)$ denote the restriction of $V$ to $\fg_J$.  

The Weyl group $W=W_I$ is generated by $\{s_i \mid i\in I\}$ subject to the relations:
\begin{enumerate}
    \item $s_i^2=1$, 
    \item $s_is_j=s_js_i$ if  $i \not\sim j$,
    \item $s_is_js_i=s_js_is_j$ if  $i\sim j$.
\end{enumerate}
Relations (2) and (3) are known as the braid relations, and they are the defining relations of the braid group $B=B_I$ with generators $\{\sigma_i \mid i\in I\}$.

Recall that the Weyl group naturally acts on the weight lattice. Let $w_I \in W_I$ be the longest element.  This element induces a diagram automorphism $\theta_I:I \to I$, given by $w_I\cdot \alpha_i =-\alpha_{\theta(i)}$. For a subdiagram $J \subseteq I$, we regard $W_J \subseteq W_I$ as a parabolic subgroup.  For a representation $U$ of $W_I$, we let $\Res_J(U)$ denote its restriction to $W_J$.

\subsection{Crystals}

We can canonically associate a combinatorial object called an $I$-crystal to any representation of $\fg$.  For us, an $I$-crystal consists of a finite set $\SS$ together with ``Kashiwara operators'' $\tilde{e_i},\tilde{f_i}:\SS \dashrightarrow \SS$ for all $i$ (the $\dashrightarrow$ signifies a partially defined map),  and a weight function $\wt:\SS \to X$.  These satisfy the following properties: for any $b,b' \in \SS$
\begin{enumerate}
    \item $\tilde{e_i}(b)=b'$ if and only if $b=\tilde{f_i}(b')$,
    \item if $\tilde{e_i}(b)$ is defined then $\wt(\tilde{e_i}(b))=\wt(b)+\alpha_i$, and  if $\tilde{f_i}(b)$ is defined then $\wt(\tilde{f_i}(b))=\wt(b)-\alpha_i$.
\end{enumerate}
Since we will only be interested in ``normal'' crystals, i.e. crystals arising from finite dimensional representations of $\fg$, no additional data is necessary to define the crystal (the functions usually denoted $\varepsilon_i,\varphi_i$ are already determined).

Kashiwara originally discovered the theory of crystals \cite{Kash91}, and showed how to construct a crystal associated to any integrable representation of the  quantum group.  There are other means of constructing crystals, for example via algebraic geometry \cite{Kam10} or category theory (which we recall below) \cite{Lau-Vaz}.  

An important point shared by all these approaches is that they are closely related to the construction and existence of canonical bases of various types, and the underlying set of the crystal is  in natural bijection with such bases. 
In particular, the crystal $\SS$ associated to the representation $V$ satisfies the property that $\dim(V_\mu)=\#\{ b\in \SS \mid \wt(b)=\mu\}$.  

There are some basic operations on crystals which mirror analogous operations on representations.  
For example, the direct sum $\SS\oplus\SS'$ is a crystal whose underlying set is obtained by disjoint union, and the additional data is defined in the obvious way.  Additionally, given a subdiagram $J\subseteq I$, the restriction of an $I$-crystal $\SS$ is a $J$-crystal denoted $\Res_J(\SS)$.  The underlying set of $\Res_J(\SS)$ is the same as $\SS$, but we forget the Kashiwara operators indexed by $i \in I\setminus J$, and the weight of $b \in \Res_J(\SS)$ is the restriction of $\wt(b)$ to $\fh_J$.  Given $I$-crystals $\SS$ and $\SS'$ one can form their tensor product $\SS\otimes\SS'$ and external tensor product $\SS\boxtimes\SS'$. In both cases, the result is an $I$-crystal whose underlying set is $\SS\times\SS'$ (we refer the reader to \cite{crystals} for the precise definitions).  

We let $\LL(\la)$ be the crystal of $L(\la)$.  Note that there are unique highest and lowest weight elements in $\LL(\la)$, which we denote $b_\la^\text{high}$ and $b_{\la}^\text{low}$.  
The crystal associated to the representation $V$ is   isomorphic to $\bigoplus_{\la \in X^+}\LL(\la)^{\oplus m(\la)}$, where $m(\la)=\dim(\Hom_\fg(L(\la),V))$.  We let $\Iso_\la(\SS):=\LL(\la)^{\oplus m(\la)}$.  This is the crystal analogue of an isotypic component.

We can picture an $I$-crystal as an $I$-labelled directed graph, whose vertices are given by the underlying set of $\SS$, with an $i$-labelled edge from $b$ to $b'$ if $\tilde{f_i}(b)=b'$.  
See Figure \ref{F:crystal} an example in type A. Note that the crystal graph of $\Res_J(\SS)$ is obtained from the crystal graph of $\SS$ by deleting all edges labelled by $i \in I \setminus J$.

We recall the generalised Sch\"utzenberger involution $\xi_I:\SS\to\SS$, which will play an important role in our work.  To define this set map, it suffices to consider a simple crystal $\LL(\la)$, since as we described above, $\SS$ is isomorphic to a direct sum of simple crystals.  
The map $\xi_I:\LL(\la)\to\LL(\la)$ is determined by the following properties:
\begin{enumerate}
    \item $\wt(\xi_I(b))=w_I\cdot \wt(b)$,
    \item $\xi_I(\tilde{e_i}(b))=\tilde{f}_{\theta_I(i)}(\xi_I(b))$,
    \item $\xi_I(\tilde{f_i}(b))=\tilde{e}_{\theta_I(i)}(\xi_I(b))$.
\end{enumerate}
Note that by (1), $\xi_I$ swaps $b_\la^{\text{high}}$ and $b_\la^{\text{low}}$, and $\xi_I$ is then determined on the rest of the crystal by properties (2) and (3). Although $\xi_I$ is not a map of crystals, it is clear that its square is a crystal morphism, and hence must be the identity map. 

We now have the map $\xi_I:\SS \to \SS$, and so for any $J \subseteq I$ we also get maps $\xi_J:\Res_J(\SS) \to \Res_J(\SS)$.  Since these are just set maps, we can write $\xi_J:\SS \to \SS$ without confusion.

\subsection{Categorical representation theory}\label{sec:catrepth}

Categorical representations will play an important role for us. 
The foundations of this theory are due to Rouquier \cite{Rou2KM} and Khovanov--Lauda \cite{KaLu}.  Although there are some differences in their approach, later work of Cautis--Lauda \cite{CL} and Brundan \cite{Bru2KM} has clarified their relationships, and the latter, in particular, proves that their definitions agree.  

In this work we will use categorical representations of $U(\fg)$ following Rouquier's approach.  We note that everything we do has analogues in the graded setting, i.e. for the quantum group and the associated Hecke algebra, but that will not be our focus here.

Let $[\cA]_\CC$ denote the complexified Grothendieck group of an abelian category $\cA$, and recall that an exact functor $F:\cA \to \cA'$ induces a linear map $[F]:[\cA]_\CC \to [\cA']_\CC$.  Recall also that a Serre subcategory $\cA' \subset \cA$ is a full subcategory such that for any exact sequence $A \to B \to C$ in $\cA$, if $A,C \in \cA'$ then $B \in \cA'$.

Let $\bk$ be the ground ring.  A categorical representation of $U(\fg)$ consists of a family of $\bk$-linear abelian categories $\cC_\mu$ where $\mu \in X$, exact linear functors $\sE_i\sone_\mu:\cC_\mu \to \cC_{\mu+\alpha_i}$ and $\sF_i\sone_\mu:\cC_\mu \to \cC_{\mu-\alpha_i}$, and natural transformations:
\begin{enumerate}
    \item $x:\sE_i\sone_\mu \longrightarrow \sE_i\sone_\mu$,
    \item $t:\sE_i\sE_j\sone_\mu \longrightarrow \sE_j\sE_i\sone_\mu$,
    \item $\eta:\sone_\mu \longrightarrow \sF_i\sE_i\sone_\mu$,
    \item $\varepsilon:\sE_i\sF_i\sone_\mu \longrightarrow \sone_\mu$.
\end{enumerate}
These satisfy relations, which depend on some further choices of scalars, which we won't directly use and so don't specify.  Details can be found in \cite{Rou2KM, Bru2KM}.  In practice, we write $\cC$ for a categorical representation, leaving the remaining data implicit.  

There are three important consequences of these relations which we mention:
\begin{enumerate}
    \item There are functorial isomorphisms 
    \begin{align*}
        &\sE_j\sF_i\sone_\mu \cong \sF_i\sE_j\sone_\mu \text{ if } i \neq j , \\
        &\sE_i\sF_i\sone_\mu \cong \sF_i\sE_i\sone_\mu \oplus \sone_\mu^{\oplus \langle h_i,\mu\rangle} \text{ if } \langle h_i,\mu\rangle \geq 0,\\
        &\sE_i\sF_i\sone_\mu \oplus \sone_\mu^{\oplus -\langle h_i,\mu\rangle} \cong \sF_i\sE_i\sone_\mu  \text{ if } \langle h_i,\mu\rangle \leq 0
    \end{align*}

\item The functors $(\sE_i\sone_\mu,\sF_i\sone_{\mu+\alpha_i})$ are bi-adjoint.  One adjunction is part of the definition (the unit and counits are given by $\eta$ and $\varepsilon$), and the other is a consequence \cite[Theorem 4.3]{Bru2KM}.   

\item The natural transformations $x$ and $t$ define an action of the quiver Hecke (alias KLR) algebra on compositions of the $\sE_i\sone_\mu$ functors.  
\end{enumerate} 
The first point above implies that from a categorical representation $\cC$ one can produce a classical representation.  
This is a representation on the complex $X$-graded vector space $V=[\cC]_\CC:=\bigoplus_\mu [\cC_\mu]_\CC$.  Define linear maps $\pi(e_i):V \to V$ by 
$$
\pi(e_i)([Z])=[\sE_i\sone_\mu(Z)],
$$
for $Z \in \cC_\mu$, and similarly $\pi(f_i):V \to V$.  Letting $\pi(h_i)$ be the diagonal  operator which acts by the scalar $\langle h_i , \mu \rangle$ on $[\cC_\mu]_\CC$, the resulting representation is given by $\pi:U(\fg) \to \End(V)$.  We say that $\cC$ categorifies the representation $V$.

The third point can be used to show that there is an isomorphism of functors $\sE_i^n\sone_\mu \cong (\sE_i^{(n)}\sone_\mu)^{\oplus n!}$, where $\sE_i^{(n)}\sone_\mu$ is a subfunctor   of $\sE_i^{n}\sone_\mu$.  The functor $\sE_i^{(n)}\sone_\mu$ is therefore called the divided-power of $\sE_i$.

Now, suppose $\cC$ is a categorical representation which categorifies a representation $V$.  We recall that the crystal of $V$ can be extracted from $\cC$ (\cite[Proposition 5.20]{CR} and \cite[Theorem 4.31]{Brundan-Davidson}).  The construction proceeds as follows: the underlying set $\SS$ of the crystal is the set of isomorphism classes of simple objects in $\cC$.  Then for $[L] \in \SS$, set $\wt([L])=\mu$ for $L \in \cC_\mu$, and define the Kashiwara operators by $\tilde{e}_i([L])=\text{socle}(\sE_i(L))$ and $\tilde{f}_i([L])=\text{socle}(\sF_i(L))$ (recall that the socle of an object in an abelian category is its maximal semisimple subobject). 

If $\cC$ categorifies an isotypic representation, we say that $\cC$ is an isotypic categorification.  
Analogous to the isotypic decomposition of a representation, any categorical representation $\cC$ has an isotypic filtration \cite[Section 3.3]{HLLY}:
\begin{Definition}\label{def:I-filt}
    An  $I$-filtration of $\cC$ is a filtration by Serre subcategories
$$
0=\cC(0) \subset \cC(1) \subset \cdots \subset \cC(n)=\cC
$$
satisfying the following properties: 
\begin{enumerate}
    \item For every $i$, $\cC(i) \subset \cC$ is a categorical subrepresentation of $\fg_I$, i.e. it is invariant under all the data defining the categorical representation on $\cC$.
    \item The  representation on the Serre quotient $\cC(i)/\cC({i-1})$ is an isotypic categorification. 
\end{enumerate}
An  $I$-filtration is ordered if, letting $\la_i \in X^+$ denote the highest weight of $\cC(i)/\cC({i-1})$, the list $\la_1,\hdots,\la_n$ consists of distinct highest weights such that if $\la_i \prec_I \la_j$ then $i<j$.
\end{Definition}

Note that in an  $I$-filtration of $\cC$, different subquotients can be isotypic categorifications of the same highest weight, but this cannot happen in an ordered  $I$-filtration.  

Any categorical representation has an ordered $I$-filtration.  
Indeed, by \cite[Remark 3.9]{HLLY} an ordered $I$-filtration can be constructed from $\SS$.  Suppose that $\SS$ decomposes into  components 
$$
\SS=\Iso_{\la_1}(\SS) \sqcup \cdots \sqcup \Iso_{\la_n}(\SS),
$$
where $\la_1,\hdots,\la_n \in X^++$ are distinct dominant integral weights.  Further, we arrange the weights so that  $\la_i \prec_I \la_j$ implies $i < j$.  Let $\cC(i)$ be the Serre subcategory of $\cC$ generated by simple objects $L$ such that $[L] \in \Iso_{\la_j}(\SS)$ for $j \leq i$.  Then  $0 \subset \cC(1) \subset \cC(2) \subset \cdots \subset \cC(n)=\cC$ is an ordered $I$-filtration of $\cC$.

Now we move on to consider triangulated equivalences that arise from categorical representations.  Let $D^b(\cC)$ denote the bounded derived category.  For $i\in I$ the Rickard complex $\Theta_i :D^b(\cC_\mu) \to D^b(\cC_{s_i(\mu)})$ was constructed by Chuang and Rouquier, who initially used it to prove Brou\'e's abelian defect conjecture for symmetric groups \cite{CR}. By definition, $\Theta_i$ is a complex of functors concentrated in non-positive degrees:
$$
\Theta_i=\big(\cdots \longrightarrow (\Theta_i)^{-1} \longrightarrow (\Theta_i)^{0}\big),
$$
where for $r\geq0$ and $\mu_i:=\langle h_i,\mu \rangle$,
\begin{equation*}
(\Theta_i)^{-r}= \begin{cases}
   \sE_i^{(-\mu_i+r)}\sF_i^{(r)}\sone_\mu &\text{ if } \mu_i \leq 0 \\ 
   \sE_i^{(\mu_i+r)}\sF_i^{(r)}\sone_\mu &\text{ if } \mu_i \geq 0
\end{cases}
\end{equation*}
The differential $(\Theta_i)^{-r} \longrightarrow (\Theta_i)^{-r+1}$ is defined using the counit $\sE_i\sF_i\sone_\mu \longrightarrow \sone_\mu$ arising from the adjunction.  The fundamental results about Rickard complexes are:
\begin{enumerate}
    \item For every $\cC, i\in I$, and $\mu \in X$, $\Theta_i:D^b(\cC_\mu) \to D^b(\cC_{s_i(\mu)})$ is an equivalence of triangulated categories \cite{CR}.
    \item The Rickard complexes satisfy the braid relations, i.e. $\Theta_i\Theta_j \cong \Theta_j\Theta_i$ if $i \not\sim j$, and $\Theta_i\Theta_j\Theta_i \cong \Theta_j\Theta_i\Theta_j$ if $i \sim j$ \cite[Section 6]{CK}.
\end{enumerate}
Therefore, for any $w \in W$, we can define  $\Theta_w := \Theta_{i_1}\circ \cdots \circ \Theta_{i_\ell}$, where $w=s_{i_1}\cdots s_{i_\ell}$ is any reduced expression.  This is well-defined up to isomorphism.  

\section{Perverse equivalences}\label{sec:3}

\subsection{Definition and properties}\label{sec:pervdefs}
We begin with the set-up for the definition of a perverse equivalence.  
Let $\cA$ be an abelian category, and let $\bB$ be the set of simple objects in $\cA$ up to equivalence.  For any collection of objects $\bB' \subseteq \cA$, we let $\langle \bB' \rangle \subseteq \cA$ be the Serre subcategory generated by $\bB'$.  Recall that $\langle \bB' \rangle$ is the smallest full subcategory of $\cA$ containing $\bB'$ which is closed under quotients and extensions.

Given a subcategory $\cB \subseteq \cA$, let $D^b_{\cB}(\cA) \subseteq D^b(\cA)$ denote the subcategory consisting of complexes with cohomology supported in $\cB$.  Equivalently, $D^b_{\cB}(\cA)$ is the thick subcategory of $D^b(\cA)$ generated by $\cB$.  We set $D^b_{\bB'}(\cA):=D^b_{\langle \bB' \rangle}(\cA)$.

Now suppose we have a chain of abelian categories $\cC \subseteq \cB \subseteq \cA$, and the induced chain of triangulated categories:
$$
\cT=D^b_{\cC}(\cA) \;\subseteq\;  \cS=D^b_{\cB}(\cA) \;\subseteq\; \cR=D^b(\cA).
$$
Let $Q:\cS \to \cS/\cT$ denote the  quotient functor, and let $(\cR^{\leq0},\cR^{\geq0})$ denote the natural t-structure on $\cR$.  We have induced t-structures on $\cS$ and $\cS/\cT$ given by $(\cR^{\leq0}\cap \cS,\cR^{\geq0} \cap \cS)$ and $(Q(\cR^{\leq0}\cap \cS),Q(\cR^{\geq0} \cap \cS))$.  Note that the heart of $\cS$ is $\cB$ and the heart of $\cS/\cT$ is $\cB/\cC$.

Now suppose we have  two chains of subsets of $\bB$:
\begin{align}
\bB(1) \subseteq \bB(2) \subseteq \cdots\subseteq \bB(s)&=\bB \\
\bB'(1) \subseteq \bB'(2) \subseteq \cdots\subseteq \bB'(s)&=\bB
\end{align}
Setting $\cA(i):=\langle \bB(i) \rangle$ and $\cT(i):=D^b_{\cA(i)}(\cA)$ and similarly for $\cA'(i)$ and $\cT'(i)$, we have induced chains of categories:
\begin{align*}
&\cA(1) \subseteq \cA(2) \subseteq \cdots\subseteq \cA(s) \\
&\cT(1) \subseteq \cT(2) \subseteq \cdots\subseteq \cT(s)
\end{align*}
and similarly for the primed versions.  Set $\cT:=D^b(\cA)$.  Note that $\cT=\cT(s)=\cT'(s)$.

\begin{Definition}
Suppose $F:\cT \to \cT$ is an autoequivalence.  Let $p:\{1,\hdots,s\} \to \ZZ$.  Then $F$ is a perverse equivalence with respect to $(\cA(\bullet),\cA'(\bullet),p)$ if for every $i$:
\begin{enumerate}
\item $F(\cT(i)) = \cT'(i)$, and
\item the induced equivalence $F[-p(i)]:\cT(i)/\cT(i-1) \to \cT'(i)/\cT'(i-1)$ is t-exact.
\end{enumerate}
\end{Definition}

Given a perverse equivalence $F$, and carrying the notation from the definition, we obtain by Condition (2) abelian equivalences $\cA(i)/\cA(i-1) \to \cA'(i)/\cA'(i-1)$ for every $i$.  Since abelian equivalences must map simple objects to simple objects, these yield bijections $$\bB(i) \setminus \bB(i-1) \leftrightarrow \bB'(i) \setminus \bB'(i-1).$$  Patching these  together we obtain a bijection $\varphi_F:\bB \to \bB$. 

The functor $F$ induces an endomorphism $f:[\cA]_\CC \to [\cA]_\CC$. For $x,y\in\BB$, write $x\leq y$ if $x\in\BB(i)\setminus\BB(i-1)$ and $y\in\BB(j)\setminus\BB(j-1)$ for $i\leq j$. 

\begin{Lemma}\label{lem:lot-general}
The operator $f$ acts on $([\cA]_\CC,\BB,\leq)$ by $\varphi_F$ up to lower-order terms.
\end{Lemma}

\begin{proof}
Suppose $x \in \bB(i) \setminus \bB(i-1)$.  Then in $[\cA]_\CC$ we have an equality
$$
f(x) = \pm \varphi_F(x) + \sum_{x' \in \bB'(i-1)}a_{x'}x',
$$
where $a_{x'} \in \ZZ$.  Note that $\varphi_F(x) \in \bB'(i) \setminus \bB'(i-1)$, and since by construction $\varphi_F(\bB(i-1))=\bB'(i-1)$, we can rewrite this as
$$
f(x) = \pm \varphi_F(x) + \sum_{y \in \bB(i-1)}a_y\varphi_F(y).
$$
\end{proof}

\subsection{Prior results}

Let $\cC$ be a categorical representation of $\fg$ with associated $\fg$-crystal $\SS$ (so the underlying set of $\SS$ is the isomorphism classes of simple objects in $\cC$).  We recall two results of Halacheva, Licata, Losev and Yacobi:

\begin{Theorem}[Corollary 6.7 in \cite{HLLY}]\label{thm:HLLY1}
Suppose $\cC$ is an isotypic categorical representation of type $\lambda$.  Then for any weight $\mu \in X$, the derived equivalence $\Theta_{w_I}[-\mathrm{ht}(w_I(\lambda)-\mu)]:D^b(\cC_\mu) \to D^b(\cC_{w_I(\mu)})$ is t-exact.  
\end{Theorem}

Let $\cC(1) \subset \cdots \subset \cC(s)=\cC$ be an  $I$-filtration of $\cC$ such that $\cC(i)/\cC(i-1)$ categorifies an isotypic representation of type $\lambda_i$.  For $\mu \in X$, let $\cC_\mu(i)=\cC_\mu \cap \cC(i)$.  Recall that the Sch\"utzenberger involution $\xi_I:\SS \to \SS$ satisfies the property that $\xi_I(\SS_\mu)=\SS_{w_I(\mu)}$.

\begin{Theorem}[Theorem 6.8 in \cite{HLLY}]\label{thm:HLLY2}
 The autoequivalence $\Theta_{w_I} :D^b(\cC_\mu) \to D^b(\cC_{w_I(\mu)})$ is a perverse equivalence with respect to $(\cC_\mu(\bullet),\cC_{w_I(\mu)}(\bullet),p)$, where $p(i)=\mathrm{ht}(w_I(\lambda_i)-\mu)$.  Moreover, the bijection between $\SS_\mu$ and $\SS_{w_I(\mu)}$ induced by $\Theta_{w_I}$ is equal to $\xi_I$.
\end{Theorem}

Fix $J \subseteq I$, and let $\cD$ denote the restriction of $\cC$ to $\fg_J$.  That is, $\cD=\cC$ as categories, but $\cD$ is regarded as a categorical representation of $U(\fg_J)$.  The weight subcategories are $\cD_\gamma$ for $\gamma \in X_J$, where $\cD_\gamma=\bigoplus_\mu \cC_\mu$ and the sum ranges over all $\mu \in X$ such that $\mu|_{\fh_J}=\gamma$.

Suppose  $\cD(1) \subset \cdots \subset \cD(s)=\cD$ is a $J$-filtration of $\cD$ such that $\cD(j)/\cD(j-1)$ categorifies an isotypic representation of type $\kappa_j$.  For $\gamma \in X_J$ define $\cD_\gamma(j)$ as above, and for $\mu \in X$ let $\cC_\mu^J(j)=\cD(j) \cap \cC_\mu$.  Recall that $w_J$ is the longest element of $W_J \subseteq W$.

\begin{Corollary}\label{cor:HLLY}
    The autoequivalence $\Theta_{w_J} :D^b(\cC_\mu) \to D^b(\cC_{w_J(\mu)})$ is a perverse equivalence with respect to $(\cC_\mu^J(\bullet),\cC_{w_J(\mu)}^J(\bullet),p)$, where $p(j)=\mathrm{ht}(w_J(\kappa_j)-\gamma)$ and $\gamma=\mu|_{\fh_J} \in X_J$.  Moreover, the associated bijection is equal to $\xi_J$.   
\end{Corollary}

\begin{proof}
    We know that $F:=\Theta_{w_J}:D^b(\cC_\mu)\to D^b(\cC_{w_J(\mu)})$ is an equivalence, and we must show it is a perverse equivalence.  By Theorem \ref{thm:HLLY2}, we have that $G:=\Theta_{w_J}: D^b(\cD_\gamma) \to D^b(\cD_{w_J(\gamma)})$ is a perverse equivalence with respect to $(\cD_\gamma(\bullet),\cD_{w_J(\gamma)}(\bullet),q)$, where $p(j)=\mathrm{ht}(w_J(\kappa_j)-\gamma)$, and the associated bijection is $\xi_J:\res_J(\SS)_\gamma \to \res_J(\SS)_{w_J(\gamma)}$.

    Now, note that 
    \begin{align*}
    D^b_{\cD_\gamma(j)}(\cD_\gamma )\cap D^b(\cC_\mu) = D^b_{\cC_\mu^J(j)}(\cC_\mu).
    \end{align*}
    This follows from an easy fact: if $\cA,\cB,\cC$ are abelian subcategories of some larger, unspecified abelian category, with $\cC \subseteq  \cB$, then $D^b_{\cA \cap \cB}(\cB)\cap D^b(\cC) =D^b_{\cA \cap \cC}(\cC)$. Therefore the data defining the perverse equivalence $G$ restricts to define perversity data for $F$ (with the same perversity function) as claimed.  Furthermore, the associated bijection is the same as $\xi_J$ above, but the domain and codomain are suitably shrunk to $\SS_\mu$ and $\SS_{w_J(\mu)}$.
\end{proof}

The case where $\mu=0$ is of most interest to us, so henceforth we set $\cA=\cC_0$ and $\bB=\SS_0$.  For the longest element $w_I$ we set $\Theta_I=\Theta_{w_I}$, and more generally $\Theta_J=\Theta_{w_J}$ for any $J \subseteq I$.

Note that in the set up of Theorem \ref{thm:HLLY2}, $\Theta_{I}$ is perverse with respect to $(\cA(\bullet),\cA(\bullet),p)$, i.e. the filtrations in the domain and codomain agree.  Moreover the perversity function in this case simplifies to $p(i)=\mathrm{ht}(\lambda_i)$.  The analogous fact holds for any $J \subseteq I$. 

\subsection{Perversity \texorpdfstring{of $\Theta_w$}{} for separable elements}\label{sec:pervysepy}

\subsubsection{} \label{sec:prep}

In this section we prove some technical results in preparation for the study of $\Theta_w$ for separable elements $w \in W$.  

We consider the following set-up: $\cA$ is an abelian category and 
$$
\cD \subset \cC \subset \cB \subset \cA
$$
is a chain of Serre subcategories.  We suppose we have an equivalence $F:D^b(\cA) \to D^b(\cA)$ which satisfies perversity-like conditions with respect to the filtration $\cD \subset \cB \subset \cA$.  Specifically,  we have:
\begin{enumerate}
\item $F(D^b_\cB(\cA))=D^b_\cB(\cA)$,
\item $F(D^b_\cD(\cA))=D^b_\cD(\cA)$, and 
\item the induced autoequivalence of $D^b_\cB(\cA)/D^b_\cD(\cA)$ (which we  denote by $\ol{F}$) is t-exact up to shift $b \in \ZZ$.
\end{enumerate}
Notice that $F$ is not actually perverse since we don't require any exactness on $D^b_\cD(\cA)$ nor on the quotient $D^b(\cA)/D^b_\cB(\cA)$.

Let $\bB_\cB$ (respectively $\bB_\cC, \bB_\cD$) be the set of simple objects up to equivalence of $\cB$ (respectively $\cC, \cD$).  Let $\varphi:\bB_\cB\setminus \bB_\cD \to \bB_\cB\setminus \bB_\cD$ be the bijection induced by (3).  Define 
$$\cC':=\left\langle \cD, \varphi(L) \;|\; L \in \bB_\cC\setminus \bB_\cD \right\rangle.$$

Let $Q:D^b_\cB(\cA) \to D^b_\cB(\cA)/D^b_\cD(\cA)$ be the natural quotient functor.  Note that the restriction of $Q$ to $\cB$ is the natural quotient functor $\cB \to \cB/\cD$.

\begin{Lemma}\label{lem:Cprime}
The category $\cC'$ is the full subcategory of $\cB$ with objects 
$$\left\{ B \in \cB \;|\; Q(B) \in \ol{F}[b](\cC/\cD) \right\}.$$
\end{Lemma}

\begin{proof}
Denote by $\cC''$ the full subcategory of $\cB$ defined in the statement of the lemma.  We first show that $\cC''$ is a Serre subcategory.  Consider an exact sequence $A_1 \to A_2 \to A_3$ where $A_1, A_3 \in \cC''$.  Applying $Q$ we obtain $Q(A_1) \to Q(A_2) \to Q(A_3)$.  Since $\ol{F}[b]$ restricts to an autoequivalence of $\cB/\cD$, we have that $\ol{F}[b](\cC/\cD)$ is a Serre subcategory of $\cB/\cD$.  By hypothesis $Q(A_1), Q(A_3) \in \ol{F}[b](\cC/\cD)$, and therefore so is $Q(A_2)$.  This shows that $A_2 \in \cC''$.  

Since both $\cC'$ and $\cC''$ are Serre subcategories,  they are determined by the simple objects they contain.  Therefore to prove the lemma it suffices to show that they contain the same simple objects.  Clearly they both contain $\bB_\cD$.  If $L$ is a simple object in $\cC'$ not  in $\cD$, then $L=\varphi(L_1)$ for some $L_1 \in \bB_\cC \setminus \bB_\cD$.  Then $Q(L) \cong \ol{F}[b](Q(L_1))$ and so $L \in \cC''$.  Conversely, if  
$L$ is a simple object in $\cC''$ not  in $\cD$ then $Q(L) \cong  \ol{F}[b](Q(L_1))$ for some simple object in $\cC /\cD$, and so $L=\varphi(L_1) \in \cC'$.

\end{proof}

\begin{Proposition}\label{prop:prep}
In the above set-up, $F(D^b_\cC(\cA)) = D^b_{\cC'}(\cA)$.
\end{Proposition}

\begin{proof}
It suffices to show that $F(D^b_\cC(\cA)) \subseteq D^b_{\cC'}(\cA)$, since the same argument will also show that $F^{-1}(D^b_{\cC'}(\cA)) \subseteq D^b_{\cC}(\cA)$.

We first show that if $X \in \cC$ then $F(X) \in D^b_{\cC'}(\cA)$.  Fix some $i \in \ZZ$ and consider the functor $D^b_\cB(\cA) \to \cB/\cC'$ given by $R \circ H^i$, where $R:\cB \to \cB / \cC'$ is the natural quotient functor and $H^i$ is the standard cohomology functor.  Note that $R \circ H^i (Y)=0$ for any $Y \in D^b_{\cC'}(\cA)$.  Therefore, by the universal property of Verdier quotients, there is a functor $G:D^b_{\cB}(\cA) / D^b_{\cC'}(\cA) \to \cB /\cC'$ making the following commute:
\[\begin{tikzcd}D^b_{\cB}(\cA) \arrow[rr] \arrow{d}{R \circ H^i} && D^b_{\cB}(\cA) / D^b_{\cC'}(\cA) \arrow{lld}{G} \\
\cB /\cC'\end{tikzcd}\]
(In fact, $G$ is given by the $i$-th cohomology with respect to the t-structure which is induced on $D^b_{\cB}(\cA) / D^b_{\cC'}(\cA)$, but we do not require this fact.)

Note that by (1),  $F(X) \in D^b_{\cB}(\cA)$.  Applying the quotient functor to $D^b_\cB(\cA)/D^b_\cD(\cA)$, we have that $Q(F(X)[b]) \cong \ol{F}[b](Q(X))$.  Then by Lemma \ref{lem:Cprime} we have $F(X)[b] \in \cC'$, implying that $F(X) \in D^b_{\cC'}(\cA)$ as desired.  

For $X \in \cA$ let $\ell(X)$ be the Jordan--H\"older length of $X$, i.e. the length of its Jordan--H\"older filtration.  
For $X \in D^b(\cA)$ let
$$
\ell(X)=\sum_i \ell(H^i(X)).
$$
Now suppose $X \in D^b_\cC(\cA)$.  We will show by induction on $\ell:=\ell(X)$ that $F(X) \in D^b_{\cC'}(\cA)$.

If $\ell=1$  then, up to isomorphism, $X$ is concentrated in one degree and the claim follows from the previous paragraphs.  
Now suppose $\ell >1$.  Consider first the case where  there exists $r$ such that $\ell(\tau^{\leq r}(X)), \ell(\tau^{\geq r+1}(X)) < \ell$.  Then applying $F$ to the exact triangle
$$
\tau^{\leq r}(X) \to X \to \tau^{\geq r+1}(X) \to
$$
we obtain
$$
F(\tau^{\leq r}(X)) \to F(X) \to F(\tau^{\geq r+1}(X)) \to
$$
and from the long exact sequence in cohomology we get exact sequences
$$
H^i(F(\tau^{\leq r}(X))) \to H^i(F(X)) \to H^i(F(\tau^{\leq r}(X)))
$$
for every $i$.  Now note that if $H^i(\tau^{\leq r}(X))$ is nonzero then it is isomorphic to $H^i(X)$, and similarly for $H^i(\tau^{\geq r+1}(X))$.  Therefore $H^i(\tau^{\leq r}(X)), H^i(\tau^{\geq r+1}(X)) \in \cC$, i.e. $\tau^{\leq r}(X), \tau^{\geq r+1}(X) \in D^b_{\cC}(\cA)$.  By inductive hypothesis therefore we have that $H^i(F(\tau^{\leq r}(X))), H^i(F(\tau^{\geq r+1}(X))) \in \cC'$.  Since $\cC'$ is a Serre subcategory we obtain from the above exact sequence that $H^i(F(X)) \in \cC'$ as well, i.e. that $F(X) \in D^b_{\cC'}(\cA)$.

Finally if there is no $r$ such that $\ell(\tau^{\leq r}(X)), \ell(\tau^{\geq r+1}(X)) < \ell$, then $X$ is a two-term complex with zero differential.  Hence it is isomorphic to a direct sum of two one-term complexes, and the result follows from the work above.
\end{proof}

\subsubsection{}\label{sec:mainthm}
Consider a chain of diagrams $\calZ=(I_1,\hdots,I_r)$ where $I \supseteq I_1 \supseteq \cdots \supseteq I_r$.  Let $w \in W$ be the corresponding separable element: $w=w_\calZ:=w_{I_1}\cdots w_{I_r}$.  In this section we will prove that $\Theta_w$ is a perverse equivalence.  For this we need a filtration of $\cC$ which is compatible with $\calZ$.

We set up up some notation for working with filtrations. 
Let $[0,n]=\{0,\hdots,n\}$ and suppose we have a filtration a filtration $\cC(\bullet)$ of $\cC$ with terms indexed by $[0,n]$:
\[0=\cC(0)\subset\cC(1)\subset\cdots\subset\cC(n)=\cC.\]
For a subset $A:=\{i_0,\dots,i_\ell\}$ of indices with $0=i_0<i_1<\dots<i_\ell=n$, we denote by $\cC_A(\bullet)$  the coarsening of $\cC(\bullet)$ given by 
\[0=\cC(i_0)\subset\cC(i_1)\subset\cdots\subset\cC(i_\ell)=\cC.\]
Let $\cC_A(k):=\cC(i_k)$ denote the $k$-th term of $\cC_A$.  
There is an induced filtration on each subquotient $\cC_A(k)/\cC_A(k-1)$ by
\[0=\cC(i_{k-1})/\cC(i_{k-1})\subset\cC(i_{k-1}+1)/\cC(i_{k-1})\subset\cdots\subset\cC(i_k)/\cC(i_{k-1})=\cC_A(k)/\cC_A(k-1).\]

\begin{Definition}\label{def:I-filts}
A filtration $\cC(\bullet)$ of $\cC$ is a \emph{$\calZ$-filtration} if it can be coarsened to an $I_1$-filtration $\cC_A$, and if $r>1$, the induced filtration on each $\cC_A(k)/\cC_A(k-1)$ is an $(I_2,\dots,I_r)$-filtration.

A filtration of $\cA$ (the zero weight subcategory of $\cC$) induced by an $\calZ$-filtration is a \emph{$\calZ$-filtration} of $\cA$.
\end{Definition}

\begin{Lemma}
    Let $\cC$ be a categorical representation of $U(\fg)$, and let $\calZ=(I_1,\hdots,I_r)$ be a chain of diagrams. Then a $\calZ$-filtration of $\cC$ exists.
\end{Lemma}

\begin{proof}
We construct a $\calZ$-filtration inductively as follows: If $r=1$ then a $\calZ$-filtration is given in \cite[Remark 3.8]{HLLY}.  Let $r>1$ and suppose we've constructed a $\calZ'$-filtration of any categorical representation of $U(\fg)$, where $\calZ'$ is a chain of length $r-1$. 
Let $$0=\cC(0) \subseteq \cC(1) \subseteq \cdots \subseteq \cC(s)=\cC$$ be an  $I_1$-filtration of $\cC$.  Then for any $i$ we have that $\cC(i)/\cC(i-1)$ is a categorical representation of $\fg_{I_1}$.  
Let $\calZ':=(I_2,\hdots,I_{r})$, and consider $\cC(i)/\cC(i-1)$ as a categorical representation of $\fg_{I_2}$ by restriction.  By hypothesis we have a $\calZ'$-filtration of $\cC(i)/\cC(i-1)$. This induces a chain of categories
$$
\cC(i-1)=\cC(i-1,0)\subseteq \cC(i-1,1) \subseteq \cdots \subseteq \cC(i-1,d_i)=\cC(i).
$$
Concatenating these chains together yields a $\calZ$-filtration of $\cC$.
\end{proof}

All the pieces are now in place to define the perversity data for $\Theta_w$.  
Let $\cC(\bullet)$ be a $\calZ$-filtration of $\cC$ (and hence induces a $\calZ$-filtration of $\cA$). Fix $1 \leq a \leq r$ and find $A=\{i_0,\dots,i_\ell\}$ so that the coarsening $\cC_A(\bullet)$ is an $I_a$-filtration.  By Corollary \ref{cor:HLLY}, $\Theta_{w_{I_a}}:D^b(\cA) \to D^b(\cA)$ is perverse with respect to the filtration $0=\cA(i_0) \subseteq \cA(i_1) \subseteq \cdots \subseteq \cA(i_\ell)=\cA$ on the domain and codomain.

By construction each $\cA(i)$ is a Serre subcategory of $\cA$.  Therefore we can define $\bB(i) \subseteq \bB$ by  $\cA(i) := \big< \bB(i) \big>$.  Now define $$\bB'(i)=\big< \xi_{I_1} \circ \cdots  \circ \xi_{I_r}\big(\bB(i)\big) \big>,$$ and let $\cA'(i):=\big< \bB'(i) \big>$.

To the subquotient $\cA(i) / \cA(i-1)$  in the $\calZ$-filtration we can associate a sequence of dominant weights  $\lambda^i_a \in X^+_{I_a}$ ($a=1,\hdots,r$) determined by:
\begin{align*}
\big[ \cA(i) / \cA(i-1) \big]_\CC \subseteq \Iso_{\lambda_a^i}(\Res_{\fg_{I_a}}(V)).
\end{align*}
Set $p(i):=\sum_{a=1}^r \mathrm{ht}(\lambda^i_a)$.

\begin{Theorem}\label{thm:mainthm}
The autoequivalence $\Theta_w$ of $D^b(\cA)$ is perverse with respect to $(\cA(\bullet),\cA'(\bullet),p)$.  Moreover, the bijection $\varphi_{\Theta_w}$ is equal to $\xi_{\calZ}=\xi_{I_1}\circ\cdots\circ\xi_{I_r}$.
\end{Theorem}

\begin{proof}
We will prove the theorem by induction on $r$. If $r=1$ and $w=w_{I_1}$, then by Corollary \ref{cor:HLLY}, $\Theta_w$ is a perverse equivalence with respect to the $I_1$-filtration of $\cA$ on both the domain and codomain.  Moreover, the perversity function is given by the height of the highest weight corresponding to each subquotient.  This agrees with the perversity data in the statement of the theorem.  Note that Corollary \ref{cor:HLLY} also gives that  the associated bijection is $\xi_{I_1}$, completing the base case.
 
For $r>1$, let $\cT(i)$ and $\cT'(i)$ be the triangulated categories defined  from $\cA(i)$ and $\cA'(i)$ as in Section \ref{sec:pervdefs}.  Write $w=w_{I_1}u$ where $u=w_{I_2}\cdots w_{I_r}$.  Set $\calZ^u=(I_2,\hdots,I_r)$.  Note that the $\calZ$-filtration of $\cA$ is in particular also a $\calZ^u$-filtration.  Let $\bB^u(i):=\big< \xi_{I_2} \circ \cdots  \circ \xi_{I_r}\big(\bB(i)\big) \big>$, and define the associated categories $\cA^u(i), \cT^u(i)$ and the perversity function $p^u(i):=\sum_{a=2}^r \mathrm{ht}(\lambda^i_a)$.
By induction we have that $\Theta_u$ is perverse with respect to $(\cA(\bullet),\cA^u(\bullet),p^u)$, and $\varphi_{\Theta_u}=\xi_{I_2}\circ\cdots\circ\xi_{I_r}$.   

Suppose $$0=\cA(i_0) \subseteq \cA(i_1) \subseteq \cdots \subseteq \cA(i_t)=\cA$$ is a coarsening of the $\calZ$-filtration to a $I_1$-filtration.  First observe that $\cA^u(i_\ell)=\cA(i_\ell)$ for $1\leq \ell \leq t$.  Indeed, for any $1 \leq a \leq r$ we have that $\varphi_{I_a}(\bB(i_\ell))= \bB(i_\ell)$.  Therefore $\Theta_u(\cT(i_\ell))=\cT(i_\ell)$.  Since we also have that $\Theta_{I_1}(\cT(i_\ell))=\cT(i_\ell)$, it follows that the same is true for $\Theta_w$.

Now take $i_{\ell-1} < j < i_\ell$ and consider the chain $\cA(i_{\ell-1}) \subset \cA^u(j) \subset \cA(i_\ell) \subset \cA$.  We are in the setting of Section \ref{sec:prep}, where $F=\Theta_{I_1}$.  By Proposition \ref{prop:prep}, we deduce that $\Theta_{I_1}(\cT^u(j)) = \cT'(j)$.  Since we know by hypothesis that $\Theta_u(\cT(j)) = \cT^u(j)$, we conclude that $\Theta_w(\cT(j)) =\cT'(j)$.  This shows that $\Theta_w$ preserves the filtrations in the perversity data.

We next show that $\Theta_w$ is t-exact up to shifts (prescribed by $p$) on the subquotients of the filtration.  Take $j$ as in the previous paragraph.  By inductive hypothesis the functor
$$
\Theta_u: \cT(j) /\cT(j-1) \to \cT^u(j)/\cT^u(j-1)
$$
is t-exact up to shift $-p^u(j)$.  Moreover, by Proposition \ref{thm:HLLY2} the functor 
$$
\Theta_{I_1}: \cT(i_\ell)/\cT(i_\ell-1) \to  \cT(i_\ell)/\cT(i_\ell-1)
$$
is t-exact up to shift $-\mathrm{ht}(\lambda^j_1)$.  The restriction of this functor to  $\cT^u(j)/\cT^u(j-1)$ induces an equivalence 
$
\cT^u(j)/\cT^u(j-1) \to \cT'(j)/\cT'(j-1)
$
which is also t-exact up to shift $-\mathrm{ht}(\lambda^j_1)$.  Therefore the composition $\Theta_w=\Theta_{I_1} \circ \Theta_u$ induces an equivalence $\cT(j) /\cT(j-1) \to \cT'(j)/\cT'(j-1)$ which is t-exact up shift $-p(j)=-\mathrm{ht}(\lambda^j_1)-p^u(j)$.

Notice that this also shows that the induced bijection $\BB(j)\setminus\BB(j-1) \to \BB'(j)\setminus\BB'(j-1)$ is  given by $\xi_{\calZ}$, since in the composition of t-exact functors (up to shift) $\Theta_{I_1} \circ \Theta_u$, $\Theta_u$ gives rise to $\xi_{I_2}\circ\cdots\circ\xi_{I_r}$ and $\Theta_{I_1}$ gives rise to $\xi_{I_1}$.
\end{proof}
	
\section{Combinatorial applications}\label{sect:crystal}

 \subsection{The general set-up}\label{sect:crystal1} We are now in position to deduce combinatorial consequences from Theorem \ref{thm:mainthm}.  First we describe a general schema for extracting  combinatorial information about the action of separable elements on canonical bases.  We will then apply this to study two situations: the Kazhdan--Lusztig basis of Specht modules in Section \ref{sect:KL}, and the  canonical bases in the zero weight space of tensor product representations in Section \ref{sect:tensor-product}.  
 
To describe the general schema, let $U$ be a representation of $W$ equipped with a  basis $\bB \subset U$.  We suppose that the pair $(U,\bB)$ arises via categorification in the following sense: there is a categorical representation $\cC$ of $U(\fg)$, and an isomorphism of $W$-modules
 \begin{align}\label{eq:Wmods}
 U \cong [\cC_0]_{\CC},
 \end{align}
 such that the set $\bB$ is identified with the set of isomorphism-classes of simple objects in $\cC_0$.
 
 We note that the $W$-module structure on $[\cC_0]_{\CC}$ is a consequence of the $U(\fg)$-module structure on $[\cC]_\CC$, and moreover, the $W$-action is categorified by the Rickard complexes. Indeed, the action of $s_i$ on the zero weight space of a $U(\fg)$-module is well-known to be given by 
 \begin{align}\label{eq:actofs_i}
     s_i=\sum_{k \geq 0}(-1)^k e_i^{(k)}f_i^{(k)},
 \end{align}
where $x^{(k)}=x^k/k!$ is the divided-power.

\begin{Remark}
    It is possible to develop our theory also for other weight spaces, and deduce results about the action of a cover of the Weyl group (the ``Tits group'') on canonical basis elements in nonzero weight spaces.  We've opted to focus on the zero weight space for simplicity.
\end{Remark}

\subsection{The action of a separable element\texorpdfstring{ on $\BB$}{}}\label{sec:sep-elts} Let $\SS$ be the $I$-crystal arising from the categorical representation $\cC$.  Recall that the crystals which arise in this manner are ``normal'', i.e.\ they are isomorphic to a direct sum of irreducible crystals $\LL(\lambda)$ (for $\lambda\in X_{I}^+$).
For $b\in\SS$, we let $\SS[b]$ be the connected component of $\SS$ containing $b$ and let $\tau_I(b)\in X_{I}^+$ be the  highest weight of $\SS[b]$, i.e.\ $b \in \Iso_{\la}(\SS)$ if and only if $\tau_I(b)=\la$.

We begin by defining a preorder $\leq_{I}$ on $\SS$ (and hence on $\BB=\SS_0$), whose linearisations will correspond to $I$-filtrations. We then extend this to chains $\calZ$.

\begin{Definition}\label{def:leq_I}
For $x,y\in\SS$, let $x\leq_I y$ if and only if $\tau_I(x)\preceq_I\tau_I(y)$. We call this the $I$-preorder on $\SS$.
\end{Definition}

Note that two elements are equivalent in this preorder if and only if they lie in the same maximal isotypic component of $\SS$, and we write $x\sim_I y$. For $J\subseteq I$  we  define $\leq_{J}$ on $\SS$ by checking the relation in $\Res_{J}\SS$. 

Fix a chain $\calZ=(I_1,\dots,I_r)$ with $I\supseteq I_1\supseteq \cdots \supseteq I_r$.

\begin{Definition}\label{def:leq_Z}
If $r=1$, define $\leq_{\calZ}$ equivalently to $\leq_{I_1}$. Otherwise, for $x,y\in\SS$, set $x\leq_{\calZ}y$ if and only if $x<_{I_1}y$, or $x\sim_{I_1}y$ and $x\leq_{(I_2,\dots,I_r)}y$.
\end{Definition}

Note that $x<_{\calZ}y$ if and only if $x <_{I_a} y$ for the minimal $a$ such that $x \not\sim_{I_a} y$.  If $x\sim_{I_a}y$ for all $1\leq a \leq r$, then $x$ and $y$ are equivalent under $\leq_{\calZ}$ and we write $x\sim_{\calZ}y$. 

The $\sim_{\calZ}$ equivalence classes are labelled by sequences $\bla=(\la_1,\hdots,\la_r)\in X_{I_1}^+\times\cdots\times X_{I_r}^+$, where $b\in \SS$ is in the equivalence class labelled by $(\tau_{I_1}(b),\hdots,\tau_{I_r}(b))$.

Let $\SS(\bla)$ denote the $\bla$-equivalence class in $\SS$, and let 
\begin{align}\label{eq:Zeqcl}
    \SS=\SS(\bla_1)\sqcup \cdots \sqcup \SS(\bla_n)
\end{align}
be the decomposition into equivalence classes.  

For $\bla$ as above, for $1\le a\le r$, set $\bla|_a=(\la_1,\hdots,\la_a)$.  
\begin{Definition}\label{def:nor}
    The list of $\sim_{\calZ}$ equivalence classes $(\bla_1,\hdots,\bla_n)$ is \emph{normally ordered} if 
    \begin{enumerate}
        \item $\bla_i <_{\calZ} \bla_j \Longrightarrow i<j$, and
        \item if $i<j$ and $(\bla_i)|_a=(\bla_j)|_a$ then $(\bla_i)|_a=(\bla_k)|_a$ for all $i\leq k\leq j$.
    \end{enumerate}
\end{Definition}

A normal ordering can be constructed by restricting the crystal $\SS$ along $\calZ$, and arranging the isotypic components along the way. More precisely, first consider the isotypic decomposition of $\SS$ regarded as an $I_1$-crystal $\SS=\Iso_{\la_1}(\SS)\sqcup \cdots \sqcup \Iso_{\la_N}(\SS)$, with the dominant weight arranged so that if $\la_i \prec_{I_1} \la_j$ then $i<j$.  Then for each component, consider its decomposition as an $I_2$-crystal:
\begin{align*}
\Iso_{\la_i}(\SS)=\Iso_{\la_i,\mu_1}(\SS)\sqcup \cdots \sqcup \Iso_{\la_i,\mu_{N_i}}(\SS),
\end{align*}
where $\Iso_{\la_i,\mu_k}(\SS):=\Iso_{\mu_k}(\Iso_{\la_i}(\SS))$ and the weights arranged so that if $\mu_k \prec_{I_2} \mu_\ell$ then $k<\ell$.  Continuing in this way down the chain $\calZ$, at the final step we obtain a normally ordered list of the $\sim_\calZ$ equivalence classes.

Let $\cC(i)$ be the Serre subcategory of $\cC$ generated by $\SS(\bla_1),\hdots,\SS(\bla_i)$, and let $\cC(\bullet)$ denote the associated filtration.

\begin{Lemma}\label{lem:ordering-to-filtration}
Let $\cC$ be a categorical representation of $U(\fg)$ with  crystal $\SS$.  Let $(\bla_1,\hdots,\bla_n)$ be a normally ordered list of the $\sim_\calZ$ equivalence classes of $\SS$, and let $\cC(\bullet)$ be the resulting filtration.  Then $\cC(\bullet)$ is a $\calZ$-filtration.  
\end{Lemma}

\begin{proof}
    We prove this by induction on $r$.  If $r=1$ then by the discussion in Section \ref{sec:catrepth},  $\cC(\bullet)$ is a $I_1$-filtration.  Suppose now $r>1$. We can break up the list \eqref{eq:Zeqcl} into segments:
\begin{align*}
    &\SS(\bla_1),\hdots,\SS(\bla_{i_1}) \\
    &\SS(\bla_{i_1+1}),\hdots,\SS(\bla_{i_2})\\
    &\hspace{1.5cm}\vdots \\
    &\SS(\bla_{i_\ell+1}),\hdots,\SS(\bla_n)
\end{align*}
where $(\bla_1)|_1=(\bla_2)|_1=\cdots=(\bla_{i_1})|_1$, $(\bla_{i_1+1})|_1=\cdots=(\bla_{i_2})|_1$, etc.  Let $A=\{i_1,\hdots,i_\ell\}$.  Then by the same discussion in Section \ref{sec:catrepth}, $\cC_A(\bullet)$ is an $I_1$-filtration.

Consider a subquotient $\cC_A(k)/\cC_A(k-1)$ of $\cC_A$ for some $1\leq k\leq \ell$. This is an isotypic categorification of highest weight $\lambda:=\bla_{i_k}|_1 \in X_{I_1}^+$.  Let $\calZ'=(I_2,\hdots,I_r)$.  The induced filtration on this subquotient is then obtained from a normal ordering of the $\sim_{\calZ'}$ equivalence classes of the crystal for $\cC_A(k)/\cC_A(k-1)$.  By induction this is a $\calZ'$-filtration, so by definition $\cC(\bullet)$ is a $\calZ$-filtration.
\end{proof}

We return to the general set-up described in Section \ref{sect:crystal1}: we have a based representation $(U,\BB)$ of $W$ which arises from the categorical representation $\cC$ of $U(\fg)$.

\begin{Theorem}\label{thm:gencomb}
Let $w \in W^\mathrm{sep}$.  Then $w$ acts on $(U,\BB,\leq_\calZ)$ by $\xi_\calZ$ up to l.o.t., where $\calZ=(I_1,\hdots,I_r)$ is a chain such that $w:=w_{I_1}\cdots w_{I_r}$.
\end{Theorem}

\begin{proof}
We consider a normal ordering of $\SS$, and the resulting filtration $\sC$ as defined above Lemma \ref{lem:ordering-to-filtration}. By the Lemma this is a $\calZ$-filtration.  Let $\cA(\bullet)$ be the associated $\calZ$-filtration of $\cA=\cC_0$. By Theorem \ref{thm:mainthm}, $\Theta_w:D^b(\cA)\to D^b(\cA)$ is a perverse equivalence with respect to $\cA(\bullet)$ and $\cA'(\bullet)$ (the latter filtration is defined above Theorem \ref{thm:mainthm}), and the associated bijection is $\xi_\calZ$.  The result now follows from Lemma \ref{lem:lot-general}.
\end{proof}

\begin{Remark}\label{rmk:cactus}
    The bijections which arise in this manner are closley related to the cactus group $C_I$ \emph{\cite{HK}}.  This is a group generated by $c_J$, where $J \subseteq I$ is a connected subdiagram, satisfying the following relations:
    \begin{enumerate}
    \item $c_J^2=1$, for all $J \subseteq I$,
    \item $c_Jc_K=c_Kc_J$ if $J \cap K =\emptyset$ and there is no edge connecting any $j \in J$ to any $k \in K$,
    \item $c_Jc_K= c_Kc_{\theta_K(J)}$ if $J \subseteq K$.
    \end{enumerate}
    The assignment $c_J \mapsto \xi_J$ defines an action of $C_I$ on $\BB$ \emph{\cite[Theorem 7.7]{HLLY}}.
\end{Remark}

\subsection{The action of separable elements on the dual basis}\label{sec:dualbasis}
In this section we describe how to obtain an analogue of Theorem \ref{thm:gencomb} for a basis dual to $\BB$.
Indeed, in  our  set-up \eqref{eq:Wmods} the basis $\BB$ is often  Lusztig's ``dual canonical basis''.  Moreover, the representation $U$ is  equipped with a non-degenerate symmetric form $(\cdot,\cdot):U \times U \to \CC$, and the basis $\BB'$ dual to $\BB$ with respect to $(\cdot,\cdot)$ is the ``canonical basis'' (cf.\ Section \ref{sect:tensor-product}).  

For example, consider the case where $U$ is the zero weight space of an irreducible representation $L(\la)$ with a fixed highest weight vector $v_\la$.  The category $\cC_0$ is equivalent to $A\mmod$ (the category of finite dimensional $A$-modules) for an appropriate algebra $A$ (cf.\ Section \ref{sect:specht-categorification}).  Then $\BB$ corresponds to simple $A$-modules, and is Lusztig's dual canonical basis. 

On the other hand, $U$ is also categorified by $A\Pmod$ (the category of finitely generated projective  $A$-modules) and in this model the projective indecomposable modules give rise to a different basis $\BB'$ (the canonical basis).  These bases are orthogonal with respect to the  symmetric bilinear form descended from the hom-pairing on $\cC_0$: for $X\in A\Pmod, Y\in A\mmod$, $([X],[Y])=\dim\Hom(X,Y)$.  Note that this is precisely the 
 Shapavalov form and satisfies the properties:
\begin{enumerate}
    \item $(v_\la,v_\la)=1$,
    \item $(e_i\cdot u,v)=(u, f_i \cdot v)$ for all $i\in I$ and any $u,v\in U$.
\end{enumerate}

To state our  result, let $V=[\cC]_\CC$ be the representation of $U(\fg)$ categorified by $\cC$ (so that $U=V_0$).
Note that we don't require $V$ to be irreducible (e.g. $V$ could be a tensor product of irreducible representations, cf.\ Section \ref{sect:tensor-product}).
Suppose that $V$ is equipped with a non-degenerate symmetric form satisfying Property (2) above.
We let $(\cdot,\cdot)$ denote the restriction of this form to $U$.

Write $\BB=\{x_\bt \mid \bt \in \cX\}$ for the basis of $U$ arising from simple objects under \eqref{eq:Wmods}, indexed by a set $\cX$.  Let $\BB'=\{x_\bt'\mid \bt \in \cX\}$ be the dual basis with respect to $(\cdot,\cdot)$.  

Given a chain of diagrams $\calZ=(I_1,\hdots,I_r)$, define: $\xi_\calZ(\bt)=\bs$ if and only if $\xi_\calZ(x_\bt)=x_\bs$, and $\bt \leq_\calZ \bs$
if and only if $x_\bt \leq_\calZ x_\bs$.

Let $w:=w_{I_1}\cdots w_{I_r}$ be the separable element corresponding to $\calZ$, and write $\xi=\xi_\calZ$.

\begin{Corollary}\label{cor:gencomb}
Let $\bt\in\cX$.  Then there exist $b_{\br}\in\ZZ$ such that
\[w\cdot x_\bt'=\pm x'_{\xi^{-1}(\bt)}+\sum_{\xi^{-1}(\br)\;>_\calZ\;\xi^{-1}(\bt)}b_{\br}x'_{\xi^{-1}(\br)}.\]
\end{Corollary}

\begin{Lemma}
    Any $w\in W$ is self-adjoint with respect to $(\cdot,\cdot):U\times U \to \CC$.
\end{Lemma}
\begin{proof}
    We need to show that for any $u,v\in U$ we have that $(w\cdot u,v)=(u,w\cdot v)$, and it suffices to prove this for simple generators $s_i \in W$.  The action of $s_i$ on $U$ can be expressed in terms of the Lie algebra action \eqref{eq:actofs_i}.
    Therefore self-adjointness follows from Property (2).
\end{proof}

\begin{proof}[Proof of Corollary \ref{cor:gencomb}]
    By Theorem \ref{thm:gencomb} we can write for any $\bs\in \cX$:
    $$
    w \cdot x_\bs =\sum_\br a_{\br\bs}x_{\xi(\br)},
    $$
    where $a_{\bs\bs}=\pm1$ and $a_{\br\bs}=0$ if $\br>_\calZ\bs$.  
    Define $b_{\br\bt} \in \ZZ$ by 
    $$
    w \cdot x'_\bt =\sum_\br b_{\br\bt}x'_{\xi^{-1}(\br)}.
    $$
    Then
    \begin{align*}
        b_{\br\bt} = (x_{\xi^{-1}(\br)},w\cdot x_\bt') = (w\cdot x_{\xi^{-1}(\br)},x_\bt') =a_{\xi^{-1}(\bt)\xi^{-1}(\br)}
    \end{align*}
Now, $a_{\xi^{-1}(\bt)\xi^{-1}(\br)}=\pm1$ if $\br=\bt$ and $a_{\xi^{-1}(\bt)\xi^{-1}(\br)}=0$ if $\xi^{-1}(\br)<_\calZ\xi^{-1}(\bt)$.  Therefore
\begin{align*}
   w \cdot x_\bt &= \pm x'_{\xi^{-1}(\bt)}+\sum_{\xi^{-1}(\br)\;>_\calZ\;\xi^{-1}(\bt)} b_{\br\bt}x'_{\xi^{-1}(\br)}.
\end{align*}
\end{proof}

\section{The Kazhdan--Lusztig Basis for the Specht Module}\label{sect:KL}

In this section, we construct the Specht module in two ways. For $\lambda\vdash n$, let $(S^\lambda,\KL_\lambda)$ be the irreducible representation for the symmetric group $S_n$ indexed by the partition $\lambda$ equipped with its Kazhdan--Lusztig basis. The classical realisation of this based representation comes from cell modules (cf.\ Section \ref{sect:KL-general}) in type A. We recall below that this basis also appears as the dual canonical basis of $L(\lambda)_0$, where $L(\lambda)$ is the irreducible $\fsl_n$-representation of highest weight $\lambda$.

We present both of these constructions, and then use Theorem \ref{thm:gencomb} to deduce the action of separable permutations on $\KL_\lambda$, and prove Theorem \ref{thm:1}. Using the combinatorics of type A crystals, we determine explicit descriptions for the bijection $\xi_\calZ$ and preorder $\leq_\calZ$ on $\KL_\lambda$ in terms of standard tableaux.

\subsection{The Specht module as a cell module}\label{sect:specht-cell}

Following \cite{BB}, we realise each Specht module as a cell module in type A. Fix $n\geq 1$ and set $I=\mathrm{A}_{n-1}$, so that $W=S_n$ is the symmetric group. See Appendix \ref{app:KL} for the general construction of cell modules. Let $\KL=\{C_w\mid w\in S_n\}$ denote the Kazhdan--Lusztig basis for $\CC[S_n]$ and let $\leq^\rL$ denote the preorder $\leq_I^\rL$ (Definition \ref{def:KL-order}) on $S_n$, and $\sim^\rL$ its induced equivalence relation. The equivalence classes under this relation are known as left cells. Each left cell $\cC$ has an associated based $S_n$-representation $(V(\cC),\KL_\cC)$ which we call a cell module of $S_n$. We call $\KL_\cC=\{[C_x]\mid x\in\cC\}$ the KL basis for $V(\cC)$.

We first give a combinatorial description of the left cells. Recall that the RSK correspondence gives a bijection
\[x\mapsto(P(x),Q(x))\]
between elements of $S_n$ and pairs of standard Young tableaux of the same shape and with $n$ boxes each. For the following two results, see Section 6.5 in \cite{BB}.

\begin{Lemma}Permutations $x,y\in S_n$ are equivalent under $\sim^\rL$ if and only if $Q(x)=Q(y)$.
\end{Lemma}

Hence, the left cells of $S_n$ are naturally indexed by standard tableaux. For a partition $\lambda\vdash n$ and $Q\in\SYT(\lambda)$, we define the left cell
\[\cC_Q:=\{w\in S_n\mid Q(w)=Q\}.\]
Write $[C_{P,Q}]$ for the unique KL basis element $[C_x]\in \KL_{\cC_Q}$ satisfying $P(x)=P$. It turns out that each cell module $V(\cC_Q)$ is an irreducible $S_n$-representation:

\begin{Lemma}\label{lem:specht}Fix $\lambda\vdash n$ and $Q,Q'\in\SYT(\lambda)$. The linear extension of the map $[C_{P,Q}]\mapsto [C_{P,Q'}]$ for every $P\in\SYT(\lambda)$ is an isomorphism of $S_n$-representations $V(\cC_Q)\xrightarrow{\sim} V(\cC_{Q'})$. Moreover, $V(\cC_Q)$ is isomorphic to the Specht module $S^\lambda$.
\end{Lemma}

Hence, we fix $Q\in\SYT(\lambda)$ and write $C_P$ for $C_{[P,Q]}$. Then $S^\lambda$ has a basis
\[\KL_\lambda:=\{C_P\mid P\in\SYT(\lambda)\}\]
which we call the KL basis, and this construction is independent of the choice of $Q$.

\subsection{The Specht module descends to the dual canonical basis}\label{sect:specht-categorification}

Consider the partition $\lambda$ as a highest weight for $\fsl_n$ and let $L(\lambda)$ denote the irreducible representation of $\fsl_n$ with this highest weight.  Recall that the Weyl group of $\fsl_n$ is $S_n$, and hence $S_n$ acts on the zero weight space $L(\lambda)_0$.  By Schur--Weyl duality we have that $L(\lambda)_0\cong S^\lambda$.

To apply Theorem \ref{thm:gencomb} we require a categorification of $L(\lambda)$.  This is provided by the cyclotomic KLR algebras, as introduced by Khovanov--Lauda and Rouquier.  These algebras can be defined either by generators and relations \cite{KL}, or via the geometry of Lusztig quiver varieties \cite{Rouq}.  We will not need the precise definition of these algebras. 

We let $\cL(\lambda)$ denote the category of their finite dimensional modules.  Specifically, in the notation of \cite{KL}, we consider the category 
$$
\cL(\lambda)=\bigoplus_{\nu} R^\lambda(\nu)\mmod,
$$
where $\nu$ runs over elements of the root lattice of $\fsl_n$.   We need the following theorems:

\begin{Theorem}Fix $\lambda\vdash n$.
\begin{enumerate}
    \item There is a categorical $\fsl_n$-action on $\cL(\lambda)$ categorifying $L(\lambda)$ \emph{\cite{KL}}.
    \item Under the isomorphism $[\cL(\lambda)]_\CC \cong L(\lambda)$, the equivalence classes of simple modules are identified with the dual canonical basis of $L(\lambda)$ \emph{\cite{Webcanon}}.
\end{enumerate}
\end{Theorem}

To complete the realisation of $(S^\lambda,\KL_{\lambda})$ via categorification we recall the following folkloric result:

\begin{Proposition}\label{prop:KLtodualcan}
Under the isomorphism of $S_n$-modules $L(\lambda)_0\cong S^\lambda$, the dual canonical basis  corresponds to the Kazhdan--Lusztig basis. 
\end{Proposition}

\begin{proof}
There are several ways to prove this; we take an explicit approach and use Skandera's  construction of the dual canonical basis via immanents \cite{Skandera}.  More precisely, by Equation (4.3) of \cite{R} the action of simple transpositions on the dual canonical basis elements in $L(\lambda)_0$ agrees with their action on KL basis elements.  It follows that these bases agree up to a global scalar.  
\end{proof}

We apply Theorem \ref{thm:gencomb} to this situation:

\begin{Theorem}\label{prop:specht-action-2}
    Fix $\lambda\vdash n$. For $w\in S_n^\mathrm{sep}$, choose a chain $\calZ=(I_1,\hdots,I_r)$ in $\mathrm{A}_{n-1}$ such that $w=w_\calZ:=w_{I_1}\cdots w_{I_r}$.  Then $w$ acts on $(S^\lambda,\KL_\lambda,\leq_{\calZ})$ by $\xi_\calZ$ up to lower-order terms.
\end{Theorem}

\subsection{Crystal combinatorics}\label{sec:TypeAcrystals}

We make Theorem \ref{prop:specht-action-2} more explicit by giving combinatorial descriptions of $\leq_{\calZ}$ and $\xi_{\calZ}$ using standard tableaux. We consider the case of a single parabolic subgroup $I'\subseteq I$, and generalise to chains via induction. These results will follow from the construction of type A crystals, and the reader can find more information in \cite{crystals}.

\subsubsection{Irreducible crystals}

If $1\leq a<b\leq n$, the subset $I_{[a,b]}:=\{a,a+1,\dots,b-1\}\subseteq I$ generates the standard parabolic subgroup $S_{[a,b]}\leq S_n$ permuting the elements $a,\dots,b$. Every nontrivial standard parabolic subgroup of $S_n$ splits into irreducibles as  $I_{[a_1,b_1]}\sqcup\cdots\sqcup I_{[a_m,b_m]}$ for $1\leq a_1< b_1<\cdots <a_m<b_m\leq n$.

For a partition $\lambda$, denote by $\LL_n(\lambda)$ the  crystal of $L(\lambda)$. The objects in this crystal are identified with $\SSYT_n(\lambda)$, the set of semistandard tableaux of shape $\lambda$ containing entries from $\{1,\dots,n\}$. The weight map is given by $\wt(T)=(\eta_1,\dots,\eta_n)$, where $\eta_i$ counts the number of occurences of boxes with label $i$ in the tableau $T$. When $\lambda\vdash n$, the zero-weight subspace $\BB(\lambda):=\LL_n(\lambda)_0$ is indexed by the standard tableaux $\SYT(\lambda)$.

The raising (respectively, lowering) operators $e_i$ (respectively, $f_i$) for $1\leq i\leq n-1$ are described by replacing a single box $i$ with $i+1$ (respectively, $i+1$ with $i$) according to a bracketing rule applied to the row-reading word of the tableau. We demonstrate this with the following example:

\begin{Example}\label{ex:lowering_operator}
Let $\lambda=(5,3,2)$ and set
\[T=\young(11223,334,45)\in \LL_5(\lambda).\]
The row-reading word is $4533411223$. To calculate the action of $f_3$, we place a closed bracket beneath every occurence of $3$, and an open bracket beneath every occurence of $4$.
\[\begin{matrix}
4 & 5 & 3 & 3 & 4 & 1 & 1 & 2 & 2 & 3 \\
( &   & ) & ) & ( &   &   &   &   & )
\end{matrix}\]
We increment the rightmost unpaired $3$. Hence:
\[\young(11223,334,45)\xrightarrow{\quad f_3\quad}\young(11223,3{{\color{blue}4}}4,45)\]
\end{Example}

The Sch\"utzenberger involution $\xi_{I}$ acts on $\LL_n(\lambda)$ by the evacuation operator for tableaux. To calculate this, we rotate the tableau by $180^\circ$, replace each entry $i$ with $n-i+1$, and rectify (that is, remove all empty inside boxes using jeu-de-taquin slides). Again, in lieu of a precise definition we provide an example:

\begin{Example}\label{ex:evac}
    We calculate the evacuation of $T$ from Example \ref{ex:lowering_operator}.
    \[\young(11223,334,45)\xrightarrow{\mathrm{rotate}}\young(:::54,::433,32211)\xrightarrow{\mathrm{replace}}\young(:::12,::233,34455)\xrightarrow{\mathrm{rectify}}
    \young(12355,234,34)\]
\end{Example}

\begin{figure}\label{F:crystal}
    \caption{The crystal graph for $\LL_4(2,2)$ with arrow labels $\color{lightgray}{\xrightarrow{1}}$, $\color{blue}{\xrightarrow{2}}$ and $\color{red}{\xrightarrow{3}}$.
    If $J=I_{[2,4]}=\{2,3\}$, then $\Res_J\LL_4(2,2)$ has three connected components. The standard tableaux are separately coloured, and are fixed by $\xi_I$.}
    \[\begin{tikzcd}
&& \young(11,22) \arrow[d,blue] && \\
&& \young(11,23) \arrow[dll,lightgray]\arrow[d,blue]\arrow[drr,red] && \\
\young(12,23)\arrow[d,blue]\arrow[drr,red] && \young(11,33)\arrow[dll,lightgray]\arrow[drr,red] && \young(11,24)\arrow[dll,lightgray]\arrow[d,blue] \\
\young(12,33)\arrow[d,lightgray]\arrow[drr,red] && \young(12,24) && \young(11,34)\arrow[dll,lightgray]\arrow[d,red] \\
\young(22,33)\arrow[d,red] &{\color{white}\young(11,11)}& {\color{teal}\young(12,34)}\arrow[dll,lightgray]\arrow[drr,red] &
{\color{teal}\young(13,24)}\arrow[dl,blue,bend left = 20, start anchor = {south}, crossing over]
\arrow[from = ul,blue, bend left = 20, end anchor = {north}, crossing over]& \young(11,44)\arrow[d,lightgray] \\
\young(22,34)\arrow[d,blue]\arrow[drr,red] && \young(13,34)\arrow[dll,lightgray]\arrow[drr,red] && \young(12,44)\arrow[dll,lightgray]\arrow[d,blue] \\
\young(23,34)\arrow[drr,red] && \young(22,44)\arrow[d,blue] && \young(13,44)\arrow[dll,lightgray] \\
&& \young(23,44)\arrow[d,blue] && \\
&& \young(33,44) &&
    \end{tikzcd}\]
\end{figure}

\subsubsection{Skew crystals}

It is useful to extend our definition of $\LL_n(\lambda)$ to allow skew shapes. If $\mu$ is a skew partition, let $\LL_n(\mu)$ be the crystal with objects indexed by $\SSYT_n(\mu)$, the skew semistandard tableaux of shape $\mu$ on labels $\{1,\dots,n\}$. The weight map and action of the operators $e_i$ and $f_i$ are defined as above. We note that $\LL_n(\mu)$ is a normal $I$-crystal, and so is isomorphic to a disjoint union of crystals each of the form $\LL_n(\lambda)$ for some partitions $\lambda$.

We recall a number of definitions from \cite{haiman}. Given a skew semistandard tableau $X$, the rectification of $X$ is the semistandard tableaux $\rect(X)$ obtained by removing all inside empty boxes via jeu-de-taquin slides. Two skew tableaux $X$ and $Y$ are called dual equivalent if and only if they have the same shape after performing any sequence of slides to the same outside or inside boxes of both tableaux, and we write $X\approx Y$. Note that $X\approx Y$ implies $\sh(X)=\sh(Y)$, where $\sh$ denotes the shape of a (skew) tableaux.

For $X\in\LL_n(\mu)$, let $\LL_n(\mu)[X]$ denote the connected component of $\LL_n(\mu)$ containing $X$.

\begin{Proposition}\label{prop:skew-crystal}
The objects of $\LL_n(\mu)[X]$ are identified with the dual equivalence class of $X$. Moreover, $\rect(-)$ induces a crystal isomorphism $\LL_n(\mu)[X]\xrightarrow{\sim}\LL_n(\nu)$, where  $\nu:=\sh(\rect(X))$.
\end{Proposition}
\begin{proof}
    For any $1\leq i\leq n$, the operator $f_i$ on tableaux commutes with jeu-de-taquin slides \cite[Theorem 3.3.1]{crystal-jdt}. Hence, $\rect(f_i\cdot X)=f_i\cdot\rect(X)$ since $\rect(-)$ can be described as a composition of slides. The same holds for $e_i$. Moreover, it is easy to see that rectification preserves the weight. Finally, this map is an isomorphism since $\rect(-)$ is a bijection between the dual equivalence class of $X$ and $\SSYT_n(\nu)$ \cite[Theorem 2.13]{haiman}.
\end{proof}

\begin{Corollary}\label{cor:skew-crystal}
The crystal $\LL_n(\mu)[X]$ has highest weight $\sh(\rect(X))$, and $\xi_I$ acts on $\LL_n(\mu)$ by sending $X$ to the unique tableau $Y$ such that $Y\approx X$ and $\rect(Y)=\ev(\rect(X))$.
\end{Corollary}

For a semistandard (skew) tableau $X$, set $\ev(X):=\xi_I\cdot X$. We can calculate $\ev(-)$ by rectifying, evacuating, and then `unrectifying' backward along the same path to preserve dual equivalence. We continue to call this map evacuation for skew tableaux, although the term \emph{reversal} is also used in the literature \cite{switching}. 

\subsubsection{Levi branching} Fix $I':=I_{[a,b]}\subseteq I$ for $1\leq a<b\leq n$ and set $\LL_n'(\lambda):=\Res_{I'}\LL_n(\lambda)$. We wish to calculate the highest-weight and action of $\xi_{I'}$ on $\LL_n'(\lambda)$, from which we can describe $\leq_{I'}$ and $\xi_{I'}$ on $\LL_n(\lambda)$.

For a (skew) tableau $T$, let $T_{[a,b]}$ denote the tableau obtained by restricting $T$ to the boxes with labels in $\{a,\dots,b\}$, and subtracting $a-1$ from each value. For example:

\[T=\young(11223,334,45)\implies T_{[2,4]}=\young(::112,223,3)\]

Observe that acting by $f_i$ on $T$ for $i\in I'$ only affects boxes with labels in $\{a,\dots,b\}$. In particular, the restriction map $(-)_{[a,b]}$ gives a crystals isomorphism
\[\LL_n'(\lambda)[T]\xrightarrow{\sim}\LL_{b-a+1}(\mu)[T_{[a,b]}],\]
where $\mu:=\sh(T_{[a,b]})$. Applying Corollary \ref{cor:skew-crystal} to the right-hand side of this crystal isomorphism gives the following:

\begin{itemize}
\item The highest-weight element of $\LL_n'(\lambda)[T]$ has weight $\tau_{I'}(T)=\sh(\rect(T_{[a,b]}))$.
\item For $R,T\in\LL_n(\lambda)$, $R\leq_{I'}T$ if and only if $\sh(\rect(R_{[a,b]}))\preceq \sh(\rect(T_{[a,b]}))$ in the dominance order on partitions.
\item For $T\in\LL_n(\lambda)$, $\xi_{I'}\cdot T$ is the unique tableau $R$ such that $R\sim_{I'}T$ and $R_{[a,b]}=\ev(T_{[a,b]})$.
\end{itemize}

Set $\ev_{I'}(T):=\xi_{I'}\cdot T$ for every $T\in\LL_n(\lambda)$. We observe that $\ev_{I'}(T)$ can be calculated by replacing $T_{[a,b]}$ with its evacuation (up to shifting by $a-1$), and preserving the rest of the tableau. This fact can also be proved using growth diagrams \cite{cpp}. We give an example below.

\begin{Example}
Set $n=5$ and $I'=I_{[2,4]}$. We have
\[\young(::1133,123,2)\quad\xrightarrow{\quad\ev\quad}\quad \young(::1122,133,3)\]
and hence
\[\young(112244,2345,355)\quad\xrightarrow{\,\,\,\,\,\ev_{I'}\,\,\,\,\,}\quad\young(112233,2445,455)\]
\end{Example}

To extend this to subdiagrams with multiple components, we first define the external products of crystals. Given a normal $I_1$-crystal $\SS_1$ and a normal $I_2$-crystal $\SS_2$, let $\SS_1\boxtimes\SS_2$ be the normal $(I_1\sqcup I_2)$-crystal with objects $x_1\boxtimes x_2$ for $x_1\in\SS_1$ and $x_2\in\SS_2$. The crystal operator $f_i$ acts by
\[f_i(x_1\boxtimes x_2)=\begin{cases}f_i(x_1)\boxtimes x_2& i\in I_1\\
x_1\boxtimes f_i(x_2) & i\in I_2\end{cases}\]
for any $i\in I_1\sqcup I_2$ and similarly for $e_i$. Moreover, $\wt(x_1\boxtimes x_2)\mapsto (\wt(x_1),\wt(x_2))\in X_{I_1}\times X_{I_2}$. From this, we observe that $\tau_{I_1\sqcup I_2}(x_1\boxtimes x_2)=(\tau_{I_1}(x_1),\tau_{I_2}(x_2))$, and the Sch\"utzenberger operator $\xi_{I_1\sqcup I_2}$ acts by $x_1\boxtimes x_2\mapsto \xi_{I_1}(x_1)\boxtimes \xi_{I_2}(x_2)$. The external product is associative and we extend the definition to finite products $\SS_1\boxtimes\cdots\boxtimes\SS_m$.

We now consider a general subset $I'=I_{[a_1,b_1]}\sqcup\cdots\sqcup I_{[a_m,b_m]}\subseteq I$. We observe that $\LL_n'(\lambda):=\Res_{I'}\LL_n(\lambda)$ has a natural decomposition into products of skew crystals.

\begin{Lemma}
Fix $I'\subseteq I$ as above and $T\in\LL_n(\lambda)$. Set $\mu_k:=\sh(T_{[a_k,b_k]})$ and $n_k:=b_k-a_k+1$ for every $1\leq k\leq m$. The map $R\mapsto R_{[a_1,b_1]}\boxtimes\cdots\boxtimes R_{[a_m,b_m]}$ gives an isomorphism of $I'$-crystals
\[\LL_n'(\lambda)[T]\xrightarrow{\sim}\LL_{n_1}(\mu_1)[T_{[a_1,b_1]}]\boxtimes\cdots\boxtimes \LL_{n_m}(\mu_m)[T_{[a_m,b_m]}].\]
\end{Lemma}

\begin{proof}
We treat $\LL_n'(\lambda)$ as a crystal with Dynkin type $\mathrm{A}_{n_1-1}\sqcup\cdots\sqcup\mathrm{A}_{n_m-1}$, where $a_k+i\in I'$ for $1\leq i<n_k$ corresponds to $i\in \mathrm{A}_{n_k-1}$. It is clear that this crystal map preserves weights. We are required to show that acting by $f_{a_k+i}$ on $\LL_n'(\lambda)$ is equivalent to acting by $f_i$ on the component $\LL_{n_k}(\mu_k)$ under the map.

Let $\rho_k$ denote the restriction map $R\mapsto R_{[a_k,b_k]}$ for any $R\in\LL_n(\lambda)$. Note that acting by $f_{a_k+i}\in I_{[a_k,b_k]}$ for $1\leq i<n_i$ only affects tableaux on the labels $a_k,\dots,b_k$. Hence:
\[\rho_lf_{a_k+i}=\begin{cases}
\rho_l & l\neq k \\
f_i\rho_k & l=k 
\end{cases}\]
as required.
\end{proof}

We again apply Corollary \ref{cor:skew-crystal}, this time on each component of the external product from the decomposition of $\LL_n'(\lambda)[T]$ in the above Lemma.

\begin{Proposition}
Fix $T\in\LL_n(\lambda)$ and $I'\subseteq I$ as above.
\begin{enumerate}
\item The connected component of $\LL_n'(\lambda)[T]$ consists of $R\in\LL_n(\lambda)$ satisfying $R_{[a_k,b_k]}\approx T_{[a_k,b_k]}$ for every $1\leq k\leq m$.
\item This component has highest-weight $(\nu_1,\dots,\nu_m)$, where $\nu_k:=\sh(\rect(T_{[a_k,b_k]}))$.
\item The Sch\"utzenberger involution $\xi_{I'}$ acts on this component by $\ev_{I_{[a_1,b_1]}}\cdots\ev_{I_{[a_m,b_m]}}$.
\end{enumerate}
\end{Proposition}

We set $\ev_{I'}(T):=\xi_{I'}(T)$ for any $T\in\LL_n(\lambda)$ as before. This can be calculated by evacuating $T_{[a_1,b_1]}$ (up to shift by $a_1-1$), evacuating $T_{[a_2,b_2]}$ (up to shift by $a_2-1$), and so on. From these results we immediately deduce our description of $\leq_{I'}$ and $\xi_{I'}$ for any $I'\subseteq I$, and combine this with Theorem \ref{prop:specht-action-2}.

For the following result, fix $n\geq 1$ and $I'=I_{[a_1,b_1]}\sqcup\cdots\sqcup I_{[a_m,b_m]}\subseteq I$ as before. For $\lambda\vdash n$ and $R,T\in\SYT(\lambda)$, set $R\leq_{I'}T$ if and only if $\sh(\rect(R_{[a_k,b_k]}))\preceq\sh(\rect(T_{[a_k,b_k]}))$ for every $1\leq k\leq m$.

\begin{Corollary}\label{cor:specht-action-3}
Let $w_{I'}$ be the permutation reversing $a_k,\dots,b_k$ for every $1\leq k\leq m$. For any $\lambda\vdash n$, associate $\SYT(\lambda)$ with $\KL_\lambda$ by $T\leftrightarrow C_T$. Then $w_{I'}$ acts on $(S^\lambda,\KL_\lambda,\leq_{I'})$ by $\ev_{I'}$ up to l.o.t.
\end{Corollary}
\begin{proof}
    The element $w_{I'}$ is the longest element of the parabolic subgroup of $S_n$ induced by $I'$. Under the association of $\KL_\lambda$ with the dual canonical basis for $L(\lambda)_0$, we have shown that $\xi_{I'}$ acts by $C_T\mapsto C_{\ev_{I'}(T)}$ and $C_R\leq_{I'}C_T$ (in the crystal ordering) if and only if $R\leq_{I'}T$. Then apply Theorem \ref{prop:specht-action-2}.
\end{proof}

If $\calZ=(I_1,\dots,I_r)$ is a descending chain, we can extend the above result to all separable permutations $w_\calZ:=w_{I_1}\cdots w_{I_r}$ by replacing $\ev_{I'}$ with $\ev_\calZ:=\ev_{I_1}\cdots\ev_{I_r}$ and $\leq_{I'}$ with $\leq_{\calZ}$ using Definition \ref{def:leq_Z}.

\subsection{Generalisation of Rhoades' Theorem and other examples}\label{sec:KLexamples}
We consider a number of special cases of Corollary \ref{cor:specht-action-3}.

\begin{Example}
When $I'=I$, $w_{I}$ is the longest element $w_0$, $\ev_{I}=\ev$ and $\leq_{I}$ is the trivial preorder. Hence, $w_0\cdot C_T=\pm C_{\ev(T)}$ for all $T\in\SYT(\lambda)$ with constant sign, recovering a well-known result proved by Graham \cite{graham}, Stembridge \cite{S} and Berenstein--Zelevinsky \cite{BZ}.
\end{Example}

\begin{Example}
Next, set $J:=I_{[1,n-1]}$ and $\calZ:=(I,J)$. Then $w_{\calZ}$ is the long cycle $c_n=(1,2,\hdots,n)$, or equivalently, the Coxeter element $s_1\cdots s_{n-1}$. It was observed by Rush \cite{rush} that $\ev_{\calZ}=\ev_{I}\ev_{J}$ is the promotion map $\pr$. The preorder $\leq_{\calZ}$ is equivalent to $\leq_{J}$, and $R\leq_{J}T$ if and only if the $n$-box occurs in a higher row of $R$ than it occurs in $T$. When $\lambda$ is a rectangular shape, the $n$-box always occurs in the same position and the order $\leq_{\calZ}$ is trivial. Hence, we have $c_n\cdot C_T=\pm C_{\pr(T)}$ for every $T\in\SYT(\lambda)$ with constant sign, obtaining Rhoades' Theorem \cite{R}. By Corollary \ref{cor:specht-action-3} we can now generalise this formula to arbitrary shapes.

\begin{Proposition}\label{prop:Rhoades}Fix $\lambda\vdash n$ and $T\in\SYT(\lambda)$. The Coxeter element $c_n$ acts on $(S^\la,\KL_\la,\leq)$ by $\pr$ up to l.o.t., where $\leq$ is the ordering given by comparing the position of the $n$-box. 
\end{Proposition}

\begin{Remark}
This result was first announced (with a combinatorial proof) by the authors at FPSAC 2022 \emph{\cite{GY}}.
\end{Remark}

We note that every power of the long cycle $c_n^k$ is a separable permutation, and so it acts on $(S^\lambda,\KL_\lambda)$ by some bijection up to lower-order terms. On rectangular tableaux, this bijection will be $\pr^k$ by Rhoades' Theorem. However, this is not true for general shapes.
\end{Example}

\begin{Example}\label{ex:GT}
Consider the chain $\calZ=(I_1,I_{2},\hdots,I_n)$  where $I_k = I_{[1,n-k+1]}$. 
This corresponds to the standard tower of symmetric groups which 
appears e.g.\ in constructions of Gelfand--Tsetlin bases of $S^\lambda$. 
The corresponding separable permutation $g_n:=w_{I_1}\cdots w_{I_n}$ is the permutations whose inverse $g_n^{-1}$ is given by (in one-line notation): 
\begin{align*}
&n\; n-2\;  \cdots \; 2\;1 \; 3\; \cdots \text{ if } n \text{ is even,} \\
&n\; n-2\;  \cdots \;1 \; 2 \; 4\; \cdots \text{ if } n \text{ is odd.}
\end{align*}
For example, $g_4=3241$ and $g_5=53124$.

The preorder $\leq_{\calZ}$ on $\SYT(\lambda)$ becomes a linear order, determined as follows: to each tableau $T$ assign the sequence $(r_n,\dots,r_1)$ with $r_i$ the row of $T$ containing $i$, and order these sequences lexicographically. For example, if $\lambda=(3,2)$ the standard Young tableaux from smallest to largest are:

\[\begin{tikzpicture}
\node (T1) at (0,0) {\young(135,24)};
\node (T2) at (2.5,0) {\young(125,34)};
\node (T3) at (5,0) {\young(134,25)};
\node (T4) at (7.5,0) {\young(124,35)};
\node (T5) at (10,0) {\young(123,45)};
\end{tikzpicture}\]
\noindent Note that this order linearises the dominance order on tableaux. The bijection $\xi_{\calZ}$ is given by an operation $\nev$ we'll call ``nested evacuation'', where $\nev(T)$ is obtained from $T$ as follows: first evacuate the smallest box, then evacuate the smallest two boxes, and so on until we evacuate the entire tableaux.  

\begin{Remark}
    Unlike the previous examples, where the bijection $\xi_{\calZ}$ recovered well-studied operations in algebraic combinatorics (namely the evacuation and promotion operators), we are not aware of nested evacuation appearing elsewhere in the combinatorics literature.
\end{Remark}

By Corollary \ref{cor:specht-action-3} we obtain:

\begin{Proposition}
Fix $\lambda\vdash n$ and $T\in\SYT(\lambda)$. The permutation $g_n$ acts on $(S^\la,\KL_\la,\leq)$ by $\nev$ up to l.o.t., where $<$ is the ordering given by comparing the position of the $n$-box, then comparing the position of the $(n-1)$-box, and so on. 
\end{Proposition}
\end{Example}

\begin{Example}
Finally, we consider the action of a single generator. Every irreducible normal crystal of type $\mathrm{A}_1$ is isomorphic to $\LL_2(\lambda)$ for some partition $\lambda=(n)$ with $n\geq 0$ (when $n=0$ we get the crystal with a single object of weight zero). The weight-zero subset of this crystal is empty when $\lambda$ has an odd number of boxes, and contains a single element otherwise. 

We return briefly to the general setup of Section \ref{sect:crystal1}. Let $W$ be a simply-laced Weyl group with Dynkin diagram $I$ and $(U,\BB)$ a based $W$-representation arising from categorification.

\begin{Proposition}\label{prop:generator-action}For $i\in I$, the generator $s_i$ acts on $(U,\BB)$ by the identity up to l.o.t.
\end{Proposition}
\begin{proof}
    Let $J\subset I$ be the subdiagram $J=\{i\}$ which has Dynkin type $\mathrm{A}_1$. Then $s_i$ is the longest element of $W_J$. Suppose $\SS$ is the underlying crystal for $(U,\BB)$ such that $\SS_0=\BB$. Then $\res_J\SS$ is a normal crystal of type $\mathrm{A}_1$, hence $\xi_J$ acts by the identity on $\BB$. Then apply Theorem \ref{thm:gencomb}.
\end{proof}

For $(S^\lambda,\KL_\lambda)$ when $\lambda\vdash n$ and $J=\{i\}\subseteq\mathrm{A}_{n-1}$, we can explicitly determine the order $\leq_J$. For a tableau $T\in\SYT(\lambda)$, write $i\in\Des(T)$ if and only if $i$ is a descent of $T$, i.e. the box containing $i+1$ occurs in a strictly lower row of $T$ than the box containing $i$. By considering $\rect(T_{[i,i+1]})$, one can observe that:
\[\tau_J(T)=\begin{cases}
    (1,1) & i\in\Des(T), \\
    (2,0) & i\not\in\Des(T).
\end{cases}\]
Applying Proposition \ref{prop:generator-action}, we recover a well-known fact about the action of each generator $s_i$ on $\KL_\lambda$ \cite[Equation 6.20]{BB}. For $R,T\in\SYT(\lambda)$, $C_T$ always appears in $s_i\cdot C_T$ with coefficient $\pm 1$ (depending on whether $i\in D(T)$), and if any other element $C_R$ appears with nonzero coefficient, then $i\in\Des(R)\setminus\Des(T)$.
\end{Example}

We conclude this section with an example showing that non-separable permutations may not act by a bijection up to lower-order terms on $\KL_\lambda$.

\begin{Example}\label{ex:nonseparable}
Set $\lambda=(3,1)$ and fix the ordered basis:
\[\KL_\lambda=\left(\,\,\young(123,4)\,,\,\,\young(124,3)\,,\,\,\young(134,2)\,\,\right)\]
Then the permutation $w=2413\in S_4\setminus S_4^\mathrm{sep}$ acts on $\KL_\lambda$ by
\[[w]_{\KL_\lambda}=\begin{bmatrix}0 & -1 & 1 \\ 1 & -1 & 1 \\ 1 & -1 & 0 \end{bmatrix}.\]
Since there are only two zero entries in this matrix, $w$ cannot act by a bijection up to lower-order terms for any ordering of $\KL_\lambda$ (Corollary \ref{cor:lot-permutation}).
\end{Example}

\begin{Conjecture}\label{conj:nonseparable}Choose $\lambda\vdash n$ with $\dim S^\lambda\geq 3$. Then $w$ acts on $(S^\lambda,\KL_\lambda)$ by a bijection up to l.o.t. if and only if $w\in W^{\mathrm{sep}}$.
\end{Conjecture}

\section{Canonical bases in tensor product representations}\label{sect:tensor-product}

\subsection{} Let $I$ be a finite type simply-laced Dynkin diagram, and let $W=W_I$ be the associated Weyl group.  
In this section we apply Theorem \ref{thm:gencomb} to study the action of separable elements of $W$ on the dual canonical basis elements (of weight zero) in tensor product representations of $\fg=\fg_I$.  

Recall that $X^+$ denotes the set of integral dominant weights of $\fg$, and for $\lambda \in X^+$, $L(\lambda)$ denotes the irreducible representation of $\fg$ of  highest weight $\lambda$.
Given a sequence $\ul{\la}=(\la_1,\hdots,\la_n)$ of dominant weights, we define $L(\ul{\la}):=L(\la_1)\otimes\cdots\otimes L(\la_n)$.  

These tensor product representations possess (dual) canonical bases constructed in \cite{LusTensProd, Lusbook}.  These bases have remarkable properties such as positivity and compatibility with isotypic filtrations, and important connections to the BGG category $\cO$ and quantum topology.  

Both the canonical and dual canonical basis are indexed by the crystal of the representation.  More precisely, the underlying set of $\SS(\la)$ indexes two bases of $L(\la)$: the dual canonical basis $\{x_b \mid b \in \SS(\la)\}$ and the canonical basis $\{x'_b \mid b \in \SS(\la)\}$.  This follows from \cite{GrojLus}, and  also follows from the categorical point of view. 

Indeed, as mentioned in Section \ref{sec:dualbasis}, elements of the dual canonical basis correspond to classes of simple modules in a categorication of $L(\la)$, which can be constructed via the category of finite-dimensional modules over cyclotomic KLR algebras.  If one instead considers the category of finitely-generated projective modules over the same algebras, this produces another categorification of $L(\la)$ but now the classes of indecomposable projective modules correspond to the canonical basis.  This theory is well-explained in \cite{Webcanon}.   

Consequently, the dual canonical and canonical bases of $L(\ul{\la})$ are indexed by the  crystal $\SS(\ul{\la}):=\SS(\la_1)\otimes\cdots\otimes \SS(\la_n)$, and we  denote these $\{x_b \mid b \in \SS(\ul{\la})\}$ and $\{x'_b \mid b \in \SS(\ul{\la})\}$ as above.  
Note that these are not simply the tensor product of bases, e.g. in $L(\la)\otimes L(\la')$, it is not necessarily true that $x_{b\otimes b'} = x_b \otimes x_{b'}$ for $b \in \SS(\la)$ and $b'\in \SS(\la')$.  However, we do have that $x_{b\otimes b'} = x_b \otimes x_{b'} + \text{ lower-order terms}$, where the order on tensor products of bases is defined via the partial order on weights \cite[Section 27.3]{Lusbook}.

\subsection{} 
Consider the zero weight space $U(\ul{\la}):=L(\ul{\la})_0$, which as usual we view  as a $W$-module.
We let $\BB(\ul{\la}) \subset U(\ul{\la})$ denote the set of dual canonical basis elements of $L(\ul{\la})_0$.  

\begin{Theorem}[Webster's Theorem]
The pair $(U(\ul{\la}),\BB(\ul{\la}))$ arises via categorification.  
\end{Theorem}

\begin{proof}
Webster constructs KLRW algberas $T^{\ul{\la}}_\nu$, which are generalisations of KLR algebras, in order to categorify the representations $L(\ul{\la})$.  Specifically, the category
$$
\cL(\ul{\lambda})=\bigoplus_{\nu} T^{\ul{\lambda}}_\nu\mmod,
$$
carries an action of $\fg$ (and the corresponding quantum group), and this representation categorifies $L(\ul{\la})$.  
By \cite[Corollary 8.10]{Webcanon}, if we instead consider the additive categorification given by taking the categories of projective modules above, then the classes of indecomposable projectives correspond to the canonical basis of $L(\ul{\la})$.  Then, by \cite[Corollary 2.4]{Webcanon}, we obtain that the classes of simple modules in $\cL(\ul{\lambda})$ correspond to the dual canonical basis elements of $L(\ul{\la})$.
\end{proof}

We are in the position to apply Theorem \ref{thm:gencomb}.  Consider a chain $\calZ$ of subdiagrams of $I$, and let $w=w_\calZ \in W$ be the corresponding separable element, $\leq_{\calZ}$ the partial order on $\BB(\ul{\la})$, and $\xi_{\calZ}$ the bijection of $\BB(\ul{\la})$.  Then
$w$ acts on $(U(\ul{\la}),\BB(\ul{\la}),\leq_\calZ)$ by $\xi_\calZ$ up to l.o.t.

\subsection{}
We investigate this result through some examples, and begin by considering the longest element $w=w_I$.  In the simplest case $\ul{\la}=(\la)$, i.e. when $w$ acts on the zero weight space of the irreducible representation $L(\la)$, we recover, and slightly refine, a classical result of Lusztig \cite[Corollary 5.9]{LusCanBasesArising}: for any $x \in \BB(\la)$ we have that $w_I\cdot x = \pm \xi_I(x)$, where $\xi_I:\BB(\la)\to \BB(\la)$ is  the Sch\"utzenberger involution.  Note that Lusztig's result predates the theory of crystals, and does not explicitly construct the involution.

Next, take $I=A_{n-1}$ and let $w_I$ act on tensor product representations of $\fsl_n$.  For instance, consider $\CC^n\otimes \CC^n \cong L(2\varpi_1) \oplus L(\varpi_2)$, where $\varpi_i$ denotes the $i$-th fundamental weight.    Write $\BB:=\big(\SS(\varpi_1)\otimes\SS(\varpi_1)\big)_0=\BB_1\oplus \BB_2$, where $\BB_1 \cong \SS(2\varpi_1)_0$ and $\BB_2 \cong \SS(\varpi_2)_0$, and let $\xi:\BB \to \BB$ be the Sch\"utzenberger involution.  Since $\varpi_2 <_I 2\varpi_1$, we have:
\begin{enumerate}
    \item For $b \in \BB_2$, $w\cdot x_b = \pm x_{\xi(b)}$,
    \item For $b \in \BB_1$, $w\cdot x_b = \pm x_{\xi(b)}+v_b$ where $v_b \in \Span_\ZZ(\BB_2)$.
\end{enumerate}
Note that in this small example we already exhibit  new combinatorial information which hasn't been apparent via classical methods.

More generally, consider the $d$-fold tensor product of $\CC^n$, and for simplicity suppose $n \geq d$.  By Schur--Weyl duality, $(\CC^n)^{\otimes d} \cong \bigoplus_{\la \vdash d}L(\la)\otimes S^\la$.  We write $$\BB:=\big(\SS(\varpi_1)^{\otimes d}\big)_0=\bigoplus_{\la \vdash d}\BB_\la,$$ where $\BB_\la\cong \big(\SS(\la)^{\oplus d_\la}\big)_0$ and $d_\la=\dim S^\la$.  Let $\xi:\BB \to \BB$ be the Sch\"utzenberger involution.  
Since the dominance order $\prec$ on partitions is equivalent to the weight order $<_I$, we have for every $b \in \BB_\la$:
$$
w\cdot x_b = \pm x_{\xi(b)}+v_b,\;\text{where } v_b \in \Span_\ZZ(\{\BB_\mu \mid \mu \prec \la\}).
$$

Finally, we consider an example of our results that shows  how other separable permutations act on bases.  
Set $I=A_2$ and let $L(\la)=\fg$ be the adjoint representation of $\fg=\fsl_3$, which has highest weight $\la=(2,1,0)$.  Also, let $J=A_1 \subset I$ be the  embedding of $A_1$ as the first vertex.  

We'll compare the action of the separable permutations $w=(1,3)$ and $c=(1,2,3)$ on the canonical and dual canonical basis of $(L(\la)\otimes L(\la))_0$. The crystal of $L(\la)\otimes L(\la)$ is defined on the set $\SYT(\lambda)\times \SYT(\lambda)$, and there are ten weight-zero elements:
\begin{align*}
    b_0=\text{\tiny$\young(23,3)$}\otimes \text{\tiny$\young(11,2)$},\;
    b_1=\text{$\tiny\young(22,3)$}\otimes \text{$\tiny\young(11,3)$},\;
    b_2=\text{$\tiny\young(13,2)$}\otimes \text{$\tiny\young(13,3)$},\;
    b_3=\text{$\tiny\young(12,3)$}\otimes \text{$\tiny\young(12,3)$},\; 
    b_4=\text{$\tiny\young(13,3)$}\otimes \text{$\tiny\young(12,2)$} \\
    b_5=\text{$\tiny\young(13,2)$}\otimes \text{$\tiny\young(12,3)$},\; 
    b_6=\text{$\tiny\young(12,3)$}\otimes \text{$\tiny\young(13,2)$},\;
    b_7=\text{$\tiny\young(12,2)$}\otimes \text{$\tiny\young(13,3)$},\;
    b_8=\text{$\tiny\young(11,2)$}\otimes \text{$\tiny\young(23,3)$},\;
    b_9=\text{$\tiny\young(11,3)$}\otimes \text{$\tiny\young(22,3)$}
\end{align*}
\noindent Below we list their images under $\tau_I$ and $\tau_J$ (cf.\ Section \ref{sec:sep-elts}).
\begin{center}
  \begin{tabular}{  l | c c c c c c c c c c }
             & $b_0$    & $b_1$   & $b_2$   & $b_3$   & $b_4$   & $b_5$   & $b_6$   & $b_7$   & $b_8$   & $b_9$\\\hline
    $\tau_I$ & $\la_0$  & $\la_1$ & $\la_1$ & $\la_1$ & $\la_1$ & $\la_2$ & $\la_3$ & $\la_4$ & $\la_4$ & $\la_4$\\ 
    $\tau_J$ & 0        & 0       & 0       & 2       & 2       & 2       & 2       & 0       & 2       & 4 \\
    \end{tabular}
\end{center}
where $\la_0=(0,0,0)$, $\la_1=(2,1,0)$, $\la_2=(3,3,0)$, $\la_3=(4,1,1)$ and $\la_4=(4,2,0)$.  The Hasse diagram of these weights is given by:
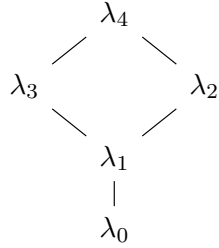
\[
\begin{tikzcd}[row sep = small, column sep = small]
 & \la_4 \ar[dl, no head] \ar[dr, no head] & \\
 \la_3 \ar[dr, no head] & & \la_2 \ar[dl, no head] \\
 & \la_1 \ar[d, no head] & \\
 & \la_0 & 
\end{tikzcd}
\]
Finally, we note that $\xi_I=(0,7)(1,3)(2,4)$ as a permutation of the $b_i$ and $\xi_J$ is trivial.

Let $x_i$ be the dual canonical basis element corresponding to each $b_i$. As an example, we consider the action of $w$ on $x_3$. Since $b_0$ is the only element smaller than $b_3$ under $<_I$,
\[w\cdot x_3=\pm x_{\xi(3)}+ax_{\xi(0)}=\pm x_1+ax_7\]
for some $a\in\ZZ$. To compute the action of $c$, note that $x_i<_{(I,J)}x_3$ for $i=0,1,2$ and so
\[c\cdot x_3=\pm x_{\xi(3)}+a_0x_{\xi(0)}+a_1x_{\xi(1)}+a_2x_{\xi(2)}=\pm x_1+a_0x_7+a_1x_3+a_2x_4\]
for $a_0,a_1,a_2\in\ZZ$. We can make similar computations to determine the actions on the canonical basis using Corollary \ref{cor:gencomb}.

\appendix

\section{Acting by bijections up to lower-order terms}\label{app:QR}

In this section we discuss a number of general results concerning linear operators which act by bijections up to lower-order terms, as introduced in Definition \ref{def:lot}.

For the following, let $U$ be a vector space over $\CC$ with a fixed basis $\BB$, and let $w:U\to U$ be a linear operator. Our first result is that if $w$ acts on $(U,\BB)$ by a bijection up to lower-order terms, this bijection is unique.

\begin{Lemma}\label{lem:lot-unique}
    For $i=1,2$, let $\xi_i:\BB\to\BB$ be bijections and let $\leq_i$ be preorders on $\BB$ such that $w$ acts on $(U,\BB,\leq_i)$ by $\xi_i$ up to lower-order terms. Then $\xi_1=\xi_2$.
\end{Lemma}
\begin{proof}
    Suppose for a contradiction that $\xi_1\neq \xi_2$, and let $x\in\BB$ be minimal under $\leq_1$ such that $\xi_1(x)\neq\xi_2(x)$. By definition, we have:
    \[w(x)=\pm\xi_1(x)+\sum_{y<_1x}a_y\xi_1(y)\]
    and we also know that $\xi_2(x)$ appears with coefficient $\pm 1$ in $w\cdot x$. Since $\xi_1(x)\neq\xi_2(x)$ by assumption, we must have $\xi_2(x)=\xi_1(y)$ for some $y<_1 x$. Since $x$ was minimal, we conclude that $\xi_1(y)=\xi_2(y)$ and hence $y=x$, a contradiction.
\end{proof}

Next, we interpret bijections up to lower-order terms through the QR decomposition of a matrix. Recall that a square matrix $A$ with real coefficients has a unique decomposition into a product of square matrices $A=QR$, where $Q$ is orthogonal and $R$ is upper-triangular with positive diagonal entries.

Suppose $(U,\BB)$ is a finite-dimensional based vector space. Given a preorder $\leq$ on $\BB$, let $\preceq$ be a linearisation of this order, which is a total order satisfying $x\prec y$ for $x,y\in\BB$ whenever $x<y$. For an operator $w:U\to U$, let $[w]_{(\BB,\preceq)}$ be the matrix of $w$ with respect to the ordered basis $(\BB,\preceq)$.

\begin{Lemma}\label{lem:QR}
    Suppose that $w$ acts on $(U,\BB,\leq)$ by $\xi$ up to lower-order terms. Then $[w]_{(\BB,\preceq)}$ has a unique $QR$ decomposition, and the orthogonal part $Q$ is a generalised permutation matrix for $\xi$.
\end{Lemma}
\begin{proof}
    Since $[w]_{(\BB,\preceq)}$ is a square matrix with integer entries, it has a unique QR decomposition. For every $x\in\BB$ we can find $a_y\in\ZZ$ such that
    \[w(x)=\pm\xi(x)+\sum_{y<x}a_y\xi(y).\]
    Define maps $\xi':U\to U$ and $w':U\to U$ by the linear extensions of $\xi'(x):=\pm\xi(x)$ and
    \[w'(x):=x\pm\sum_{y<x}a_y y\]
    for every $x\in \BB$, where the sign in both cases matches the sign of $\xi(x)$ in $w(x)$. Then $w=\xi' w'$, $[\xi']_{(\BB,\preceq)}$ is a generalised permutation matrix for $\xi$ with entries $\pm 1$, and $[w']_{(\BB,\preceq)}$ is an upper-triangular matrix whose diagonal entries are all $1$, since the order $<$ implies $\prec$ by construction. Permutation matrices with nonzero entries $\pm 1$ are orthogonal, so
    \[[w]_{(\BB,\preceq)}=[\xi']_{(\BB,\preceq)}[w']_{(\BB,\preceq)}\]
    is the QR decomposition as required.
\end{proof}

We note that a decomposition of this form is extremely rare in general, since permutation matrices with entries $\pm 1$ form a finite subset of the space of orthogonal matrices.

Given some based vector space $(U,\BB)$ and an operator $w:U\to U$, we would like to determine whether $w$ acts by \emph{some} bijection up to lower-order terms. Suppose $A:=[w]_{(\BB,\preceq)}$, where $\preceq$ is an arbitrary ordering on $\BB$.

\begin{Corollary}\label{cor:lot-permutation}
    The operator $w:U\to U$ acts by some bijection on $(U,\BB)$ up to lower-order terms if and only if there exist permutation matrices $P_1,P_2$ such that $P_1AP_2$ is an upper-triangular integer matrix with diagonal entries $\pm 1$.
\end{Corollary}
\begin{proof}
    By the Lemma, $w$ acts by a bijection on $(U,\BB)$ by a bijection up to l.o.t. if and only if there is some ordering $(\BB,\preceq')$ of the basis such that $[w]_{(\BB,\preceq')}=PR$, where $P$ is a permutation matrix and $R$ is upper-triangular with integer entries and diagonal entries $\pm 1$. Since reordering the basis corresponds to conjugating by a permutation matrix, this is equivalent to the existence of a permutation matrix $P_0$ such that
    \[P_0AP_0^{-1}=PR.\]
    Setting $P_1:=P^{-1}P_0$, $P_2:=P_0^{-1}$ and multiplying by $P^{-1}$ on both sides completes the proof.
\end{proof}

This is the problem of `upper-triangularisation': asking whether the rows and columns of a matrix $A$ can be permuted to achieve upper-triangularity (although we also impose some additional constraints on the entries). When $A$ is a binary matrix, the computational complexity of this problem has been shown to be NP-complete \cite{triangularisation}.

\section{The Kazhdan--Lusztig basis for arbitrary Coxeter groups}\label{app:KL}

In this section, we prove that separable elements of a Coxeter group act on the Kazhdan-Lusztig basis of the left regular representation by bijections up to lower-order terms. Our argument is combinatorial, and builds on work of Mathas \cite{M}, Geck \cite{G} and Roichman \cite{R} in relative Kazhdan-Lusztig theory. We note that a version of these results for representations of Hecke algebras (with equal parameters) can be similarly developed using recent work by Bonnaf\'e \cite{Bonnafe}, although this involves a number of technicalities which we choose to avoid for the sake of exposition.

For the general theory of Coxeter groups and their Kazhdan--Lusztig representations, see \cite{BB}. We begin by making the necessary definitions to state our main result.

Let $W:=(W,I)$ be a Coxeter system. That is, $I$ is a finite set, which is the vertex set of a simple, undirected graph with edges labelled $\{3,4,\dots,\infty\}$. Given such an $I$, the associated Coxeter group $W$ is generated by $s_i$ for $i\in I$ satisfying $s_i^2=1$ and $(s_is_j)^{m_{ij}}=1$, where $m_{ij}$ is the weight of the edge between $i$ and $j$. If no such edge exists, we set $m_{ij}:=2$ so that $s_i$ and $s_j$ commute. If $m_{ij}=\infty$, no relation is imposed. Note that every Weyl group can be realised as a finite Coxeter group. 

An expression for $x\in W$ is a word $(i_1,\dots,i_r)$ with $i_j\in I$ such that $x=s_{i_1}\cdots s_{i_r}$. We call this expression reduced if it is of minimal length among all expressions for $x$. The length of such an expression is called the length of $x$.

For $J\subseteq I$, we let $(W_J,J)$ be the induced  Coxeter system. Then $W_J\leq W$ is the subgroup generated by $\{s_i\mid i\in J\}$. When $W_{J}$ is a finite group, it contains a unique longest element which we denote $w_{J}$. It turns out that this element is always an involution. We use this to extend the definition of separable elements of Weyl groups from Section \ref{sec:other-weyl-groups} to all Coxeter groups by
\[W^\mathrm{sep}:=\{w_{I_1}\cdots w_{I_r}\in W\mid I\supseteq I_1\supseteq\cdots\supseteq I_r,\, W_{I_1}\text{ finite}\}.\]
In the following section, we construct a canonical basis for the regular representation $\CC[W]$, which we call the Kazhdan--Lusztig (KL) basis. Our main result will be as follows:

\begin{Theorem}\label{thm:lot-KL}
    Let $W$ be a Coxeter group and $\KL$ the KL basis of $\CC[W]$. Every $w\in W^\mathrm{sep}$ acts on $(\CC[W],\KL)$ by a bijection up to lower-order terms.
\end{Theorem}

\subsection{Setup}\label{sect:KL-general} In this section we construct the KL basis $\CC[W]$ following \cite{BB}. As before, let $W:=(W,I)$ be an arbitrary Coxeter system.

Denote by $\cH:=\cH(W;q)$ the Hecke algebra, which is the $\Z[q^{\pm 1/2}]$-algebra generated by $\{T_i\mid i\in I\}$ satisfying the relations $(T_iT_j)^{m_{ij}}=1$, and the quadratic relation $T_i^2=(q-1)T_i+q$. If $w=s_{i_1}\cdots s_{i_r}$ is a reduced expression, define $T_w:=T_{i_1}\cdots T_{i_r}$ (this is independent of the chosen reduced expression). Observe that $\cH$ is spanned over $\Z[q^{\pm 1/2}]$ by $\{T_w\mid w\in W\}$, and each $T_w$ is invertible.

There is a ring involution on $\H$ given by the $\Z$-linear extension of $q^{1/2}\mapsto q^{-1/2}$ and $T_w\mapsto T_{w^{-1}}^{-1}$ for every $w\in W$. By work of Kazhdan and Lusztig \cite{KaLu}, there is a unique basis $\{C'_w(q)\mid w\in W\}$ for $\cH$ which is invariant under this involution, and satisfies some additional properties which we won't explicitly need.

To obtain representations of $W$, we specialise this construction to $q\mapsto 1$. Observe that there is a ring isomorphism 
\[\cH(W;1)\xrightarrow{\sim}\Z[W]\]
sending $T_w$ to $w$. We let $C_w$ denote the image of $C'_w(1)$ under this isomorphism, and write $\KL=\{C_w\mid w\in W\}$ for the induced basis of the regular representation $\CC[W]=\CC\otimes_\ZZ\ZZ[W]$, which we call the Kazhdan--Lusztig (KL) basis.

The construction of a fixed basis for $\CC[W]$ gives a natural filtration of subrepresentations. We generalise this to determining a filtration of $\Res_{J}\CC[W]:=\Res_{W_{J}}\CC[W]$ for every $J\subseteq I$, whose subquotients can be given a KL basis in a natural way. When $W_{J}$ is finite, we will see that the longest element $w_{J}\in W_{J}$ acts by bijection (up to sign) on the KL basis of these subquotients.

For $x,y\in W$ and $i\in I$, write $x\xleftarrow{i}y$ if $C_x$ appears with nonzero coefficient in the expansion of $s_i\cdot C_y$ with respect to the KL basis.

\begin{Definition}\label{def:KL-order}Fix $J\subseteq I$. For $x,y\in W$, set $x\leq_{J}^\rL y$ if and only if there exists a chain
\[x=x_0\xleftarrow{i_1}x_1\xleftarrow{i_2}x_2\xleftarrow{i_3}\cdots\xleftarrow{i_r}x_r=y\]
for $x_0,x_1,\dots,x_r\in W$ and $i_1,\dots,i_r\in J$. Write $\sim_{J}^L$ for the equivalence relation induced by this preorder.
\end{Definition}

Note that $C_x$ appears with nonzero coefficient in $w\cdot C_x$ for some $w\in W_{J}$ only if $x\leq_{J}^\rL y$, and that $\leq_{J'}^\rL$ implies $\leq_{J}^\rL$ whenever $J'\subseteq J\subseteq I$. Let $\cC$ be an equivalence class of $W$ under $\leq_{J}^\rL$, which we call a $J$-relative left cell. For $x\in W$, write $x\leq_{J}^\rL\cC$ if $x\leq_{J}^\rL y$ for some (equivalently, for all) $y\in\cC$. Define the subspace
\[\cC^\leq:=\Span_\CC(\{C_x\mid x\leq_{J}^\rL\cC\})\subseteq\CC[W]\]
which is a subrepresentation of $\Res_{J}\CC[W]$. We define $x<_{J}^\rL\cC$ and $\cC^<$ similarly. Observe that $\cC^<$ is a $\Res_{J}\CC[W]$-subrepresentation of $\cC^\leq$, and so we can define the quotient representation
\[V(\cC):=\cC^\leq/\cC^<.\]
This is a $W_{J}$-representation, and we call such representations $J$-relative cell modules of $W$. For $x\in\cC$, let $[C_x]_{J}$ be the image of the KL basis element $C_x\in\cC^\leq$ under the natural projection map $\cC^\leq\twoheadrightarrow V(\cC)$. Observe that $\KL_{\cC}:=\{[C_x]_{J}\mid x\in\cC\}$ is a basis for $V(\cC)$, which we call the KL basis.

When $J=I$, we write $\leq^\rL$ (respectively $\sim^\rL$) instead of $\leq_I^\rL$ (respectively $\sim_I^\rL$), and $[C_x]$ instead of $[C_x]_I$. We often drop the prefix `$I$-relative' for cells and modules and refer simply to left cells and cell modules of $W$.

\begin{Remark}When $W=S_n$ is the symmetric group, the cell modules are irreducible (see Section \ref{sect:specht-cell}). However, cell modules of general Coxeter groups will be neither irreducible nor finite-dimensional in general.\end{Remark}

\subsection{Actions on cell modules} 
Let $(W,I)$ be an arbitrary Coxeter system. Our goal is to show that for any $J\subseteq I$ with $W_{J}$ finite, the element $w_{J}$ acts on $\CC[W]$ by an involution on the KL basis. 

We begin with the action of the longest element on cell modules for finite Coxeter groups, which forms the basis for the remainder of our calculations.

\begin{Proposition}[Theorem 3.1 in \cite{M}]\label{prop:long-element-KL}
Suppose $W$ is finite and $\cC$ is a left cell of $W$. Then there is an involution $\psi_\cC:\cC\to\cC$ such that the longest element $w_I\in W$ acts on the KL basis of $V(\cC)$ by
\[w_I\cdot [C_x]=\pm [C_{\psi_{\cC}(x)}]\]
for every $x\in\cC$, where the sign is constant.
\end{Proposition}
\begin{Remark}
In \emph{\cite{M}}, Mathas states this result only for Weyl groups, since the general case relies on the so-called `positivity conjecture' to hold in each Hecke algebra. However, this conjecture has since been proven in the case of equal parameters \emph{\cite{positivity}}.
\end{Remark}

To extend this to parabolic subgroups, we need to compare the $J$-relative left cells of $W$ to the left cells of $W_J$, which is possible due to a number of observations of Roichman \cite{roichman}.

Fix $(W,I)$ an arbitrary Coxeter system and $J\subseteq I$ a subdiagram. Every $x\in W$ has a unique decomposition $x=ua$, where $u\in W_{J}$ and $a$ is the minimal-length element of $W_{J}x\subseteq W$. We define special notation for this decomposition by setting $\Res_{J}(x):=u$ and $\pi_{J}(x):=a$.

\begin{Lemma}[Theorem 5.2 in \cite{roichman}]\label{lem:KL-restriction}
Fix $w\in W_J$, $\cC$ a $J$-relative left cell and $x,y\in \cC$. Then $[C_x]_{J}\in V(\cC)$ appears in $w\cdot [C_y]_{J}$ with nonzero coefficient only if $\pi_{J}(x)=\pi_{J}(y)$. If this is the case, the coefficient with which it appears matches the coefficient of $C_{\res_{J}(x)}\in\CC[W_{J}]$ in the action of $w\cdot C_{\res_{J}(y)}$.
\end{Lemma}

In particular, for every $J$-relative left cell $\cC$ of $W$, there is a left cell $\cC'$ of $W_J$ such that $\res_J(-):\cC\to\cC'$ is a bijection, and the map $[C_x]_J\mapsto [C_{\res_J(x)}]$ for every $x\in\cC$ linearly extends to an isomorphism $V(\cC)\to V(\cC')$ of $W_J$-represetations.

By this Lemma, we are able to deduce the action of $w_{J}$ on the KL basis of $\CC[W]$ up to lower-order terms. We first define the bijection that $w_J$ will act by. As usual, we define this bijection on $W$ and use the natural association of $x\in W$ with $C_x\in\KL$.

\begin{Definition}
Fix $J\subseteq I$ with $W_{J}$ finite. For $x\in W$, define $\psi_{J}:W\to W$ by
\[\psi_{J}(x):=\psi_{\cC}(\Res_{J}(x))\pi_{J}(x),\]
where $\cC$ is the left cell of $W_{J}$ containing $\Res_{J}(x)$ and $\psi_{\cC}$ is the involution from Proposition \ref{prop:long-element-KL}.
\end{Definition}

In particular, we note that $\psi_J$ is an involution which preserves $J$-relative left cells setwise, and hence permutes the KL basis elements of each $J$-relative cell module.

\begin{Proposition}\label{prop:long-element-KL2}
    For any $J\subseteq I$ with $W_{J}$ finite, the longest element $w_{J}\in W_{J}\leq W$ acts on $(\CC[W],\KL,\leq_{J}^\rL)$ by $\psi_{J}$ up to l.o.t.
\end{Proposition}
\begin{proof}
    Fix $x,y\in W$. Since $w_{J}\in W_{J}\leq W$, $C_x$ appears in $w_{J}\cdot C_y$ only if $x\leq_{J}^\rL y$. By the first part of Lemma \ref{lem:KL-restriction}, we have $x\sim_J^\rL y$ if and only if $\pi_{J}(x)=\pi_{J}(y)$ and $\res_{J}(x)\sim^\rL\res_J(y)$.

    Let $\cC$ be the left cell of $W_{J}$ containing $u:=\res_{J}(x)$. By Proposition \ref{prop:long-element-KL2}, $w_{J}$ acts on the KL basis element $[C_u]\in V(\cC)$ by
    \[w_{J}\cdot [C_u]=\pm [C_{\psi_{\cC}(u)}]\]
    with constant sign. Note that $\psi_{J}(x)=\psi_{\cC}(u)\pi_{J}(x)$ by definition. Hence, by applying the second part of Lemma \ref{lem:KL-restriction} we obtain that there exist $b_y\in\ZZ$ such that
    \[w_{J}\cdot C_x=\pm C_{\psi_{J}(x)}+\sum_{y<_{J}^\rL x}b_yC_y\]
    in $\CC[W]$. The sign only depends on $\pi_{J}(x)$ and the left cell of $W_J$ containing $\res_{J}(x)$, and hence is constant on equivalence classes of $\sim_{J}^\rL$. Finally, setting $a_y:=b_{\psi_{J}(y)}$ and noting that $y\sim_{J}^\rL\psi_{J}(y)$, we can reindex the sum to obtain
    \[w_{J}\cdot C_x=\pm C_{\psi_{J}(x)}+\sum_{y<_{J}^\rL x}a_yC_{\psi_{J}(y)}\]
    as required.
\end{proof}

We now extend this to separable elements. For a chain $\calZ=(I_1,\dots,I_r)$ with $I_1\supseteq\cdots\supseteq I_r$, we define $\leq_\calZ^\rL$ in analogy with Definition \ref{def:leq_Z}.

\begin{Definition}\label{D:gamma-ordering-KL}
If $r=1$, define $\leq_{\calZ}^\rL $ equivalently to $\leq_{I_1}^\rL $. Otherwise, for $x,y\in W$, set $x\leq_{\calZ}^\rL y$ if and only if $x<_{I_1}^\rL y$, or $x\sim_{I_1}^\rL y$ and $x\leq_{(I_2,\dots,I_r)}^\rL y$.
\end{Definition}

Since $\leq_{I_1}^\rL$ is implied by $\leq_{I_2}^\rL$ which is implied by $\leq_{I_3}^\rL$ and so on, we have that $x<_{\calZ}^\rL y$ for $x,y\in W$ whenever $x<_{I_a}^\rL y$ for any $1\leq a\leq r$. Similarly, if $\calZ':=(I_2,\dots,I_r)$, then $x<_{\calZ'}^\rL y$ implies $x<_{\calZ}^\rL y$ (note that the analogous statements do not hold for the crystal orderings).

For an arbitrary separable element $w\in W^\mathrm{sep}$, there is a chain $\calZ=(I_1,\dots,I_r)$ such that $W_{I_1}$ is finite and $w=w_{\calZ}:=w_{I_1}\cdots w_{I_r}$. Set $\psi_{\calZ}:=\psi_{I_1}\cdots\psi_{I_r}$, which is a bijection on $W$ preserving $I_1$-relative left cells setwise.

\begin{proof}[Proof of Theorem \ref{thm:lot-KL}]Define $\psi_\calZ$ and $\leq_\calZ^\rL$ on $\KL$ by identifying $x$ with $C_x$. We prove that $w_{\calZ}$ acts on $(\CC[W],\KL,\leq_\calZ^\rL)$ by $\psi_{\calZ}$ up to l.o.t. for any chain $\calZ=(I_1,\dots,I_r)$ with $W_{I_1}$ finite.

We prove this by induction on $r$. Proposition \ref{prop:long-element-KL2} gives the $r=1$ case. Now suppose $r>1$ and set $\calZ':=(I_2,\dots,I_r)$. By assumption $w_{\calZ'}$ acts on $\KL$ by
\[w_{\calZ'}\cdot C_x=\pm C_{\psi_{\calZ'}(x)}+\sum_{y<_{\calZ'}^\rL x}b_yC_{\psi_{\calZ'}(y)}\]
for every $x\in W$. Apply $w_{I_1}$ to both sides. On the left-hand side, we have $w_{\calZ}\cdot x$ by definition. Suppose $C_{\psi_\calZ(z)}=C_{\psi_{I_1}(\psi_{\calZ'}(z))}$ for $z\in W$ appears with nonzero coefficient after acting on the left-hand side by $w_{I_1}$. There are two cases to consider.

\begin{enumerate}
    \item $C_{\psi_\calZ(z)}$ appears in $w_{I_1}\cdot C_{\psi_{\calZ'}(x)}$ with nonzero coefficient. Then either $z=x$ (in which case $C_{\psi_{\calZ}(z)}$ appears with coefficient $\pm 1$) or $\psi_{\calZ'}(z)<_{I_1}^\rL \psi_{\calZ'}(x)$. In the latter case, we have $z<_{I_1}^\rL x$ and so $z<_{\calZ}^\rL x$.
    \item $C_{\psi_\calZ(z)}$ appears in $w_{I_1}\cdot C_{\psi_{\calZ'}(y)}$ with nonzero coefficient for some $y<_{\calZ'}^\rL x$ (that is, $y<_{\calZ}^\rL x$). Using the same argument as in (1), we observe that $z=y$ or $z<_{I_1}^\rL y$. In either case, we have $z\leq_{\calZ}^\rL y<_{\calZ}^\rL x$ as required.
\end{enumerate}

Finally, the coefficient of each $C_{\psi_\calZ(z)}$ is an integer by the construction of the KL basis, and the sign with which $C_{\psi_\calZ(x)}$ appears depends only on the $\sim_{I_1}^\rL$ and $\sim_{\calZ'}^\rL$ equivalence classes of $x$ which are uniquely determined by the $\sim_{\calZ}^\rL$ equivalence class.
\end{proof}

\begin{Remark}\label{rmk:cactus-kl}
    From this one can deduce that the cactus group of $W$ (defined for general Coxeter groups as generated by $c_J$ for $J\subseteq I$ connected and $W_J$ finite, with relations in Remark \ref{rmk:cactus}) acts on $W$ by $c_J\cdot x=\psi_J(x)$, recovering a result in \emph{\cite{Bonnafe,losev}}.
\end{Remark}

We extend this theorem to $J$-relative cell modules of $W$. For this result to make sense, $J$ must be chosen large enough so that $W_J$ contains $w_{\calZ}$.

\begin{Corollary}\label{cor:kl-action}
Fix $I\supseteq J\supseteq I_1\supseteq\cdots\supseteq I_r$ with $W_{I_1}$ is finite and set $\calZ:=(I_1,\dots,I_r)$. Let $\cC$ be a $J$-relative left cell. Then $w_\calZ$ acts by $\psi_{\calZ}$ on $(V(\cC),\KL_{\cC},\leq_{\calZ}^{\rL})$ up to lower-order terms.
\end{Corollary}
\begin{proof}
    We define $\psi_\calZ$ and $\leq_{\calZ}^\rL$ on $\KL_\cC$ under the association of $x\in\cC$ with $[C_x]_J\in\KL_\cC$. 
    
    By the proof of Theorem \ref{thm:lot-KL}, for every $x\in \cC$ we have
    \[w_\calZ\cdot C_x=\pm C_{\psi_\calZ}(x)+\sum_{y<_\calZ^\rL x}a_yC_{\psi_\calZ(y)}\]
    with $a_y\in\Z$ and a sign constant on equivalence classes under $\sim_\calZ^\rL$. Since $\psi_\calZ(y)$ preserves $\sim_{I_1}^\rL$ equivalence classes and hence $\sim_J^\rL$ equivalence classes, the above formula is also valid in the subrepresentation $\cC^\leq\subseteq\Res_J\CC[W]$. Applying the projection map $\cC^\leq\twoheadrightarrow V(\cC)$ we obtain
    \[w_\calZ\cdot [C_x]_{J}=\pm [C_{\psi_\calZ(x)}]_{J}+\sum_{y<_\calZ^\rL x}a_y[C_{\psi_\calZ(y)}]_{J}\]
    for every $x\in\cC$. Since $\psi_\calZ$ preserves $\cC$ setwise, the only nonzero elements in the above sum are of the form $[C_{\psi_{\calZ}(y)}]_J$ for $y\in\cC$. In particular, $[C_{\psi_\calZ(x)}]_{J}=\psi_\calZ([C_x]_J)$.
\end{proof}

We consider this result in the case $W=S_n$ and $J=I$. By the construction given in Section \ref{sect:specht-cell}, every cell module $V(\cC)$ is isomorphic to $S^\lambda$ for some $\lambda\vdash n$, and $\KL_\cC$ corresponds to $\KL_\lambda=\{C_T\mid T\in\SYT(\lambda)\}$. Comparing Corollary \ref{cor:kl-action} with Theorem \ref{prop:specht-action-2} and applying Lemma \ref{lem:lot-unique}, we observe that
\[\psi_{\calZ}(x)\mapsto\left(\ev_{\calZ}(P(x)),Q(x))\right)\]
under the RSK correspondence for every $x\in S_n$ and every chain $\calZ$ of $\mathrm{A}_{n-1}$. This recovers a combinatorial description of the cactus group action described in Remark \ref{rmk:cactus-kl} for type A, which can be found in \cite{cpp, white}.

\providecommand{\bysame}{\leavevmode\hbox to3em{\hrulefill}\thinspace}
\providecommand{\MR}{\relax\ifhmode\unskip\space\fi MR }
\providecommand{\MRhref}[2]{%
  \href{http://www.ams.org/mathscinet-getitem?mr=#1}{#2}
}
\providecommand{\href}[2]{#2}

\end{document}